\documentclass[12pt]{amsart}
\usepackage{etex}
\usepackage{amssymb}
\usepackage{amsxtra}
\usepackage{mathrsfs}
\usepackage{graphics}
\usepackage{latexsym}
\usepackage{xcolor}
\usepackage{amsmath}
\usepackage{amssymb,amsthm,amsfonts}
\usepackage{mathtools}
\usepackage{amscd}
\usepackage[arrow, matrix, curve]{xy}
\usepackage{syntonly}

\usepackage{tikz}
\usepackage{tikz-cd}
\usepackage{amsmath}
\usepackage{amsfonts}
\usepackage{braket}
\usepackage{adjustbox}
\usetikzlibrary{matrix, calc, arrows}

\usepackage{color}

\title[A note on HDRFs of level zero]{A note on Higgs-de Rham flows of level zero}

\author[Mao Sheng]{Mao Sheng}
\email{msheng@ustc.edu.cn}
\address{School of Mathematical Sciences,
University of Science and Technology of China, Hefei, 230026, China}

\author[Jilong Tong]{Jilong Tong}
\email{jilong.tong@cnu.edu.cn}
\address{School of Mathematical Sciences and Beijing Advanced Innovation Center for Imaging Theory and Technology, Capital Normal University, Beijing, 100048, China}

\begin{document}
\theoremstyle{plain}
\newtheorem{thm}{Theorem}[section]
\newtheorem{theorem}[thm]{Theorem}
\newtheorem*{theorem*}{Theorem}
\newtheorem{lemma}[thm]{Lemma}
\newtheorem{corollary}[thm]{Corollary}
\newtheorem*{corollary*}{Corollary}
\newtheorem{proposition}[thm]{Proposition}
\newtheorem{addendum}[thm]{Addendum}
\newtheorem{variant}[thm]{Variant}
\theoremstyle{definition}
\newtheorem{lemma and definition}[thm]{Lemma and Definition}
\newtheorem{construction}[thm]{Construction}
\newtheorem{notations}[thm]{Notations}
\newtheorem{question}[thm]{Question}
\newtheorem{problem}[thm]{Problem}
\newtheorem{remark}[thm]{Remark}
\newtheorem{remarks}[thm]{Remarks}
\newtheorem{definition}[thm]{Definition}
\newtheorem{statement}[thm]{Statement}
\newtheorem{claim}[thm]{Claim}
\newtheorem{assumption}[thm]{Assumption}
\newtheorem{assumptions}[thm]{Assumptions}
\newtheorem{properties}[thm]{Properties}
\newtheorem{example}[thm]{Example}
\newtheorem{conjecture}[thm]{Conjecture}
\newtheorem{proposition and definition}[thm]{Proposition and Definition}
\numberwithin{equation}{thm}

\newcommand{\pP}{{\mathfrak p}}
\newcommand{\sA}{{\mathcal A}}
\newcommand{\sB}{{\mathcal B}}
\newcommand{\sC}{{\mathcal C}}
\newcommand{\sD}{{\mathcal D}}
\newcommand{\sE}{{\mathcal E}}
\newcommand{\sF}{{\mathcal F}}
\newcommand{\sG}{{\mathcal G}}
\newcommand{\sH}{{\mathcal H}}
\newcommand{\sI}{{\mathcal I}}
\newcommand{\sJ}{{\mathcal J}}
\newcommand{\sK}{{\mathcal K}}
\newcommand{\sL}{{\mathcal L}}
\newcommand{\sM}{{\mathcal M}}
\newcommand{\sN}{{\mathcal N}}
\newcommand{\sO}{{\mathcal O}}
\newcommand{\sP}{{\mathcal P}}
\newcommand{\sQ}{{\mathcal Q}}
\newcommand{\sR}{{\mathcal R}}
\newcommand{\sS}{{\mathcal S}}
\newcommand{\sT}{{\mathcal T}}
\newcommand{\sU}{{\mathcal U}}
\newcommand{\sV}{{\mathcal V}}
\newcommand{\sW}{{\mathcal W}}
\newcommand{\sX}{{\mathcal X}}
\newcommand{\sY}{{\mathcal Y}}
\newcommand{\sZ}{{\mathcal Z}}
\newcommand{\A}{{\mathbb A}}
\newcommand{\B}{{\mathbb B}}
\newcommand{\C}{{\mathbb C}}
\newcommand{\D}{{\mathbb D}}
\newcommand{\E}{{\mathbb E}}
\newcommand{\F}{{\mathbb F}}
\newcommand{\G}{{\mathbb G}}
\renewcommand{\H}{{\mathbb H}}
\newcommand{\I}{{\mathbb I}}
\newcommand{\J}{{\mathbb J}}
\renewcommand{\L}{{\mathbb L}}
\newcommand{\M}{{\mathbb M}}
\newcommand{\N}{{\mathbb N}}
\renewcommand{\P}{{\mathbb P}}
\newcommand{\Q}{{\mathbb Q}}
\newcommand{\Qbar}{\overline{\Q}}
\newcommand{\R}{{\mathbb R}}
\newcommand{\SSS}{{\mathbb S}}
\newcommand{\T}{{\mathbb T}}
\newcommand{\U}{{\mathbb U}}
\newcommand{\V}{{\mathbb V}}
\newcommand{\W}{{\mathbb W}}
\newcommand{\Z}{{\mathbb Z}}
\newcommand{\g}{{\gamma}}
\newcommand{\bb}{{\beta}}
\newcommand{\as}{{\alpha}}
\newcommand{\Id}{{\rm Id}}
\newcommand{\rk}{{\rm rank}}
\newcommand{\END}{{\mathbb E}{\rm nd}}
\newcommand{\End}{{\rm End}}
\newcommand{\Hom}{{\rm Hom}}
\newcommand{\Hg}{{\rm Hg}}
\newcommand{\tr}{{\rm tr}}
\newcommand{\Sl}{{\rm Sl}}
\newcommand{\Gl}{{\rm Gl}}
\newcommand{\Cor}{{\rm Cor}}

\newcommand{\SO}{{\rm SO}}
\newcommand{\OO}{{\rm O}}
\newcommand{\SP}{{\rm SP}}
\newcommand{\Sp}{{\rm Sp}}
\newcommand{\UU}{{\rm U}}
\newcommand{\SU}{{\rm SU}}
\newcommand{\SL}{{\rm SL}}
\newcommand{\Spec}{\mathrm{Spec}}
\newcommand{\Spf}{\mathrm{Spf}}
\newcommand{\ra}{\rightarrow}
\newcommand{\xra}{\xrightarrow}
\newcommand{\la}{\leftarrow}
\newcommand{\Nm}{\mathrm{Nm}}
\newcommand{\Gal}{\mathrm{Gal}}
\newcommand{\Res}{\mathrm{Res}}
\newcommand{\GL}{\mathrm{GL}}
\newcommand{\Rep}{\mathrm{Rep}}
\newcommand{\Isom}{\mathrm{Isom}}
\newcommand{\GSp}{\mathrm{GSp}}
\newcommand{\Tr}{\mathrm{Tr}}

\newcommand{\bA}{\mathbf{A}}
\newcommand{\bK}{\mathbf{K}}
\newcommand{\bM}{\mathbf{M}} 
\newcommand{\bP}{\mathbf{P}}
\newcommand{\bC}{\mathbf{C}}
\newcommand{\HIG}{\mathrm{HIG}}
\newcommand{\MIC}{\mathrm{MIC}}
\newcommand{\Aut}{\mathrm{Aut}}
\newcommand{\Out}{\mathrm{Out}}

\newcommand{\Fil}{\mathrm{Fil}}
\newcommand{\Gr}{\mathrm{Gr}}
\newcommand{\ttr}{\mathrm{tr}}



\thanks{The first named author is supported by National Natural Science Foundation of China (Grant No. 11622109, No. 11721101) and Anhui Initiative in Quantum Information Technologies (AHY150200). The second named author is supported by One-Thousand-Talents Program of China.}

\begin{abstract} The notion of Higgs-de Rham flows was introduced by Lan-Sheng-Zuo in \cite{LSZ}, as an analogue of Yang-Mills-Higgs flows in the complex nonabelian Hodge theory. In this short note we investigate a small part of this theory, and study those Higgs-de Rham flows which are of level zero.  We improve the original definition of level-zero Higgs-de Rham flows (which works for general levels), and establish a Hitchin-Simpson-type correspondence between such objects and certain representations of fundamental groups in positive characteristic, which generalizes the classical results of Katz \cite{Katz73}. We compare the deformation theories of two sides in the correspondence, and translate the Galois action on the geometric fundamental groups of algebraic varieties defined over finite fields into the Higgs side.
\end{abstract}



\maketitle

Let $k$ be a perfect field of positive odd characteristic. Let $W:=W(k)$ denote the ring of Witt vectors, with $K$ its field of fractions and $W_n$ its reduction modulo $p^n$ for every $n\in \mathbb N$. Let $\sX$ be a geometrically connected proper smooth scheme over $W$, and set $X_n:=\sX\otimes_W W_n$. Inspired by the classical Hitchin-Simpson correspondence over the field of complex numbers, the first named author together with Lan and Zuo established in their joint work \cite{LSZ} a correspondence between certain Higgs bundles over $X$ with trivial Chern classes and certain integral crystalline representation of the \'etale fundamental group $\pi_1(\sX_K)$ of the generic fiber $\sX_K=\sX\otimes_W K$ (relative to some geometric base point). In the heart of this correspondence is the notion of \textit{Higgs-de Rham flows} introduced in loc. cit. Recall that, for $n\in \mathbb N$, a Higgs-de Rham flow on $X_n$ is, roughly speaking,  a diagram of the following form
\begin{equation*}
\xymatrix{& (H_0,\nabla_{0})\ar[rd]^{\Gr_{\Fil_{0}}}&     & (H_{1},\nabla_{1})\ar[rd]^{\Gr_{\Fil_{1}}} & \\ (E_0,\theta_0)\ar[ru]^{C_{n}^{-1}} & & (E_1,\theta_1) \ar[ru]^{C_{n}^{-1}}  & & \cdots}.
\end{equation*}
Here $(E_0,\theta_0)$ is certain Higgs bundle on $X_{n}$; $C_{n}^{-1}$ is the truncated inverse Carter transform constructed by Lan-Sheng-Zuo, which lifts the inverse Cartier transform of Ogus-Vologodsky (\cite{OV}); for each integer $i>0$, $\Fil_i=(\Fil_i^jH_i)_{j\geq 0}$ is a decreasing filtration by subbundles on the flat module $(H_i,\nabla_i)$ of level $\leq p-1$ (i.e., $\Fil_i^0H_i=H_i$ and $\Fil_i^{p-1}H_i=0$) satisfying the Griffiths transversality condition; and $(E_i,\theta_i)$, $i\geq 1$, is the graded Higgs module associated with the filtered flat bundle $(H_{i-1}, \nabla_{i-1},\Fil_{i-1})$: see Construction \ref{const:GrFil} below for an explanation. We say that the Higgs-de Rham flow above is $f$-periodic for some integer $f\geq 1$, if it is moreover equipped with an isomorphism of Higgs bundles:
\[
\phi:(E_f,\theta_f)\stackrel{\sim}{\longrightarrow} (E_0,\theta_0).
\]
Based on this notion of $f$-periodic Higgs-de Rham flows and on some previous work of Faltings, Lan-Sheng-Zuo established a correspondence between such objects on $X_n$ and the \textcolor{black}{crystalline} $W_n(\F_{p^f})$-representation of $\pi_1(\sX_K)$, from where they deduced their main Hitchin-Simpson type correspondence mentioned above.

In this short note we shall investigate a small part of the theory of Higgs-de Rham flows. We study only those periodic Higgs-de Rham flows which are \emph{of level zero}, so that the filtrations $\Fil_i$ appearing above satisfy
\[
\Fil_i^0H_i=H_i\quad \textrm{and}\quad
\Fil_i^1H_i=0 \quad \textrm{for every }i\geq 0.
\]
Our first result is a Hitchin-Simpson type correspondence for these objects, which generalizes the classical Katz's correspondence \cite[Proposition 4.1.1]{Katz73} to a global setting.

\begin{theorem*}[Theorem \ref{equivalence in HT zero case}] Let $p\geq 2$ be a prime number. Let $f>0$ be an integer, and $k$ a perfect field of characteristic $p$ containing a finite field $\F_{p^f}$ with $p^f$ elements. Let $n\in \Z_{>0}$, and \textcolor{black}{$X_n$ be a smooth connected scheme} over $W_n:=W_n(k)$. Set $X:=X_n\otimes_W k$. Then there is an equivalence of categories between the category of continuous representations of $\pi_1(X)$ on finite free $W_n(\F_{p^f})$-modules, and the category of $f$-periodic Higgs-de Rham flows of level zero over $X_n$.
\end{theorem*}

Compared with the original definition of Higgs-de Rham flows in \cite{LSZ}, in the statement above we removed the restriction on the characteristic of the field $k$ and an extra liftability condition on $X_n$: the reason is simply because here we are only interested in periodic Higgs-de Rham flows \emph{of level zero}: see \S~\ref{subsection:HDRF0} for a detailed discussion. 

\textcolor{black}{If $X$ is a smooth connected $k$-scheme which is lifted to a smooth $p$-adic formal scheme $\mathcal X$ over $W$, by a limit argument, we obtain from the theorem above a correspondence between the category of continuous representations of $\pi_1(X)$ on finite free $W(\F_{p^f})$-modules, and the category of $f$-periodic Higgs-de Rham flows of level zero over $\mathcal X$ (Corollary \ref{cor:HT-equivalence-in-formal-case}). }If moreover $\mathcal X$ is the formal completion along the closed fiber of a proper smooth $W$-scheme, we check in Proposition \ref{factor through specialization map Part 1} that our Hitchin-Simpson type correspondence above is compatible with the one for Higgs-de Rham flows of general levels given by Lan-Sheng-Zuo in \cite{LSZ}.

\begin{corollary*}[Corollary \ref{obstruction space}, Corollary \ref{deformation space}]Use the notation as above and suppose moreover that the base field $k$ is algebraically closed and that $X/k$ is proper. Assume also that $X_{n}$ can be lifted to a proper smooth $W_{n+1}$-scheme $X_{n+1}$. Let $\rho: \pi_1(X)\to \mathbf{GL}_d(\Z/p^{n})$ be a continuous representation and $(E,\Fil_{tr},\phi)$ the corresponding one-periodic Higgs-de Rham flow of level zero over $X_n$. Then the following statements hold:
\begin{itemize}
\item [(i)] $\rho$ can be deformed to a continuous $\mathbf{GL}_d(\Z/p^{n+1})$-representation if and only if $E$ can be deformed to $X_{n+1}$ as vector bundle.
\item [(ii)] Suppose that the deformation obstruction of $\rho$ over $\mathbf{GL}_d(\Z/p^{n+1})$ vanishes. Then the set of isomorphism classes of deformations of $\rho$ as a $H^1(\pi_1(X),ad(\rho))$-torsor, is naturally identified with a set of isomorphism classes of deformations of $E$ over $X_{n+1}$ as a $H^1(X,\mathcal End(\bar E))^{\Phi=1}$-torsor, where $\bar E=E\otimes k$ and $\Phi$ is the induced Frobenius on the endomorphism bundle $\mathcal End(\bar E)$ from $\phi$.
\end{itemize}
\end{corollary*}

In the theorem above, if we fix an algebraic closure $\bar k$ of $k$ and apply the result to the base change $X_n\otimes_{W_n} W_n(\bar k)$ of $X_n$ to $W_n(\bar k)$, we obtain a Hitchin-Simpson type correspondence for the representations of the geometric fundamental group $\pi_1^{geo}(X):=\pi_1(X\otimes_k\bar k)$ of $X$. On the other hand, it is well-known that $\pi_1^{geo}(X)$ is endowed naturally with an (outer) action of the absolute Galois group $\mathrm{Gal}(\bar k/k)$ of $k$. Motivated by questions arising from anabelian geometry, we would like to have an algebro-geometric description of this (outer) Galois action. When $k$ is a finite field, we achieve in \S~\ref{section:Galois-action} such a description in terms of periodic Higgs-de Rham flows of level zero:

\begin{corollary*}[Corollary \ref{main result}] Let $k\subset \bar{\mathbb F}_p$ be a finite field of $q=p^m$ elements. Let $X_n$ be a smooth connected scheme over $W_n$, with $n\in \Z_{>0}$. Set $X:=X_n\otimes_W k$. Let $(E,0,\Fil_{tr},\ldots, \Fil_{tr},\phi)$ be an $f$-periodic Higgs-de Rham flow of level zero on $X_n\hat{\otimes}_{W_n}W_n(\bar{\F}_{p})$, with $\rho$ the corresponding $W_{n}(\mathbb F_{p^f})$-representation of $\pi_1^{geo}(X)$ given by the Hitchin-Simpson type correspondence above. Let $\sigma\in \mathrm{Gal}(\bar{\mathbb F}_p/k)$ which sends $a\in \bar{\mathbb F}_p$ to $a^{q}$.
Then, the $f$-periodic Higgs-de Rham flow corresponding to $\rho\circ \tilde{\sigma}$ is $(1\otimes\sigma)^*(E,0,\Fil_{tr},\ldots,\Fil_{tr},\phi)$ shifted $Nf-m$ times, where $N\in \mathbb N$ such that $Nf-m\geq 0$. Here $\tilde{\sigma}$ is an automorphism on $\pi_1^{geo}(X)=\pi_1(X\otimes_k\bar{\F}_p)$ induced from the isomorphism $1\otimes \sigma:X_n\stackrel{\sim}{\ra}X_n$ of schemes.
\end{corollary*}

Even though we do not have any interesting application of this simple reformulation of Galois action at the moment, we feel that it may still shed light on the study of the outer Galois representations of algebraic varieties over finite fields.

{\bf Acknowledgement.} We would like to thank Professor Kang Zuo for his interest in this work. We are very grateful to the two anonymous referees for their insightful comments and helpful suggestions. Last but not least, we express our sincere thanks to our editor for his professional work.

\if false

Let $k$ be a field, and $k^{sep}$ a separable closure of $k$. Write $G_k=\mathrm{Gal}(k^{sep}/k)$ for the absolute Galois group of $k$. Let $X$ be a geometrically connected smooth variety over $k$. Set $\bar X=X\times k^{sep}$, and let $\bar x$ be a $k^{sep}$-point of $X$. Let
\begin{equation}\label{eq:outer}
\alpha=\alpha_{X/k}: G_k\to  \Out(\pi_1(\bar X, \bar x))=\Aut(\pi_1(\bar{X},\bar x))/\mathrm{Inn}(\pi_1(\bar X,\bar x))
\end{equation}
be the outer Galois representation determined by the short exact sequence of profinite groups below
$$
1\to \pi_1(\bar X, \bar x)\to \pi_1(X,\bar x)\to G_k\to 1.
$$
In the following, we shall ignore the choice of base point and write for simplicity $\pi_1^{geo}(X)=\pi_1(\bar X, \bar x)$. It is a fundamental
question in arithmetic geometry to understand the outer action $\alpha$ of $G_k$ on $\pi_1^{geo}(X)$. In this note, we propose a way to study $\alpha$, natural from the nonabelian Hodge theoretical point of view.
Let $R$ be a topological commutative ring. Let
\[
\rho:\pi_1^{geo}(X)\longrightarrow \mathrm{GL}_r(R)
\]
be an object in the category $\Rep_R(\pi_1^{geo}(X))$ of $R$-linear representations of $\pi_1^{geo}(X)$ in finite free $R$-modules. Let $\sigma\in G_k$, and $\tilde{\sigma}\in \Aut(\pi_1^{geo}(X))$ a lifting of $\alpha(\sigma)\in \Out(\pi_1^{geo}(X))$. Consider the representation
\[
\rho\circ \tilde{\sigma}: \pi_1^{geo}(X)\longrightarrow \mathrm{GL}_r(R).
\]
Since the difference between two liftings of $\alpha(\sigma)$ is an inner automorphism of $\pi_1^{geo}(X)$, the isomorphism class $[\rho\circ \tilde{\sigma}]$ of $\rho\circ \tilde{\sigma}$ does not depend on the choice of $\tilde{\sigma}$. By abuse of notation, let us denote this isomorphism class by $\rho\circ \alpha(\sigma)$.  Then, the outer Galois representation \eqref{eq:outer} induces a right action of $G_k$ by $\alpha$ on (the set of isomorphism classes of)
$\Rep_R(\pi_1^{geo}(X))$ by sending $\rho$ to
$$
\rho^{\sigma}:=\rho\circ \alpha(\sigma).
$$
On the other hand, if we are given an action of $G_k$ on $R$, i.e., a continuous morphism $\phi: G_k\to \Aut(R)$, there is an induced left action of $G_k$ on $\Rep_R(\pi_1^{geo}(X))$ by mapping $\rho$ to
\begin{equation}
^{\sigma}\rho:=\beta(\sigma)\circ \rho,
\end{equation}
where $\beta(\sigma)$ is the induced Galois action of $\phi(\sigma)$ on $\GL_r(R)$. As strongly suggested by the $p$-adic Hodge theory, one hopes that, when the pair $(R,\phi)$ as above is chosen properly, there exist some nice algebro-geometric parametrization of the category $\Rep_R(\pi_1^{geo}(X))$ (in terms of certain Higgs bundles for example) and some coupling between these two Galois actions, so that we can deduce knowledge of outer Galois representation \eqref{eq:outer} from the diagonal action
\[
\Rep_R(\pi_1^{geo}(X))\ni\rho\mapsto \sigma^{*}(\rho):={}^{\sigma^{-1}}\rho ^{\sigma}
\]
of $G_k$ on representations of the geometric fundamental groups.

Our study in the case of finite fields is based on the notion of Higgs-de Rham flows introduced in the joint work \cite{LSZ} of the first named author with G. Lan and K. Zuo. Our first result is a generalization of Katz's correspondence \cite[Proposition 4.1.1]{Katz73} to a global setting.
\begin{theorem*}[Theorem \ref{equivalence in HT zero case} and Corollary \ref{cor:HT-equivalence-in-formal-case}]
Let $n\in \Z_{>0}\cup \{\infty\}$. Let $k$ be a perfect field of characteristic $p$ and $\sX$ a smooth connected formal scheme over $W_n:=W_n(k)$. Then, for each integer $f>0$, there is an equivalence of categories between the category $\mathrm{Rep}_{W_n(\F_{p^f})}(\pi_1(\sX))$ and the category $\mathrm{HDR}_{0,f}(\sX/W_n)$ of $f$-periodic HDRFs of level zero over $\sX$.
\end{theorem*}
For a precise definition of the category $\mathrm{HDR}_{0,f}(\sX/W_n)$, we refer the readers to \S~\ref{subsection:HDRF0}. It is shown in Proposition \ref{Frobenius-periodic bundles} that when there is a Frobenius lifting over $\sX$, the category $\mathrm{HDR}_{0,f}(\sX/W_n)$ is nothing but the category of Frobenius-periodic vector bundles of period $f$.

Our main result on the outer Galois representation on the geometric fundamental group of an algebraic variety over a finite field is the following
\begin{theorem*}[Theorem \ref{main result}] Let $k\subset \bar{\mathbb F}_p$ be a finite field of $q=p^m$ elements and $X_n$ a smooth connected formal scheme over $W_n(k)$ for some $n\in \Z_{>0}\cup \{\infty\}$. Set $X=X_n\times k$. Let $(E,0,\Fil_{tr},\ldots, \Fil_{tr},\phi)$ be an $f$-periodic HDRF of level zero on $X_n$, with $\rho$ the corresponding $W_{n}(\mathbb F_{p^f})$-representation of $\pi_1^{geo}(X)$. Let $\sigma\in \mathrm{Gal}(\bar{\mathbb F}_p/k)$ be the topological generator sending $a\in \bar{\mathbb F}_p$ to $a^{q}$.
Then, the periodic HDRF corresponding to $\rho^{\sigma}$ is $\sigma^*(E,0,\Fil_{tr},\ldots,\Fil_{tr},\phi)$ shifted $Nf-m$ times, where $N\in \mathbb N$ such that $Nf-m\geq 0$.
\end{theorem*}

\fi

\if false
\section{The case of finite fields}
Let $k=\mathbb F_q\subset \bar \F_p$ be a finite field of characteristic $p$, and $\sigma\in \mathrm{Gal}(\bar \F_p/k)$ the Frobenius $k$-automorphism of $\bar \F_p$: so $\sigma(a)=a^q$ for any $a\in \bar \F_p$. Let $X$ be a geometrically connected smooth variety over $k$. In this section, we want to find an algebro-geometric description of of the (outer) action of $\sigma$ on $\pi_1^{geo}(X)$. Our main tool is the $p$-torsion Hitchin-Simpson correspondence established in \cite{LSZ}.

\fi

\section{Periodic HDRFs of level zero}\label{subsection:HDRF0} Let $k$ denote a perfect field of positive characteristic $p$. Let $n>0$ be an integer and set $W_n:=W_n(k)$. Let $X_n$ be a smooth scheme over $W_n$. Write $X_m:=X_n\otimes_{W_n}W_m$ for every $1\leq m\leq n$. In this section, we shall review the definition of periodic \emph{Higgs-de Rham flow}, written HDRF for short in the sequel, of level zero introduced in \cite{LSZ}, which will be central in our discussion. To do so, assume as in \cite{LSZ} the following condition:
\begin{itemize}
\item[(*)] \emph{$\mathrm{Char}(k)=p\geq 3$, and $X_{n}$ can be lifted to a smooth $W_{n+1}$-scheme $X_{n+1}$}.
\end{itemize}
We shall see in Remark \ref{rem:remove-extra-assumptions} that, if we are only interested in the level-zero case, the condition (*) above can be removed.

Let us start with the following construction used in \cite{LSZ}.

\begin{construction}\label{const:GrFil} Let $X/S$ be a smooth $S$-scheme, and $H$ a flat bundle on $X$: so $H$ is a vector bundle on $X$, equipped with an integrable connection
\[
\nabla:H\longrightarrow H\otimes \Omega_{X/S}^1.
\]
Let $\Fil^{\bullet}H=(\Fil^i H)_{i\in \mathbb Z}$ be a decreasing separated and exhaustive filtration by subbundles of $H$, satisfying the \textit{Griffiths transversality} condition, i.e.,
\[
\nabla(\Fil^iH)\subset \Fil^{i-1}H\otimes\Omega_{X/S}^1, \quad \forall \ i\in \mathbb Z.
\]
The filtered flat bundle $(H,\nabla,\Fil^{\bullet} H)$ is also referred to as a \emph{de Rham bundle} on $X/S$ in the following. Write $\Gr_{\Fil^{\bullet}}(H)$ for the associated graded $\mathcal{O}_{X}$-module of $H$ relative to the filtration $\Fil^{\bullet}H$, so that
\[
\Gr_{\Fil^{\bullet}}(H)=\bigoplus_{i\in \mathbb Z}\Gr_{\Fil^{\bullet}}^i(H), \quad \textrm{with }\Gr_{\Fil^{\bullet}}^i(H):=\Fil^iH/\Fil^{i+1}H.
\]
Because of the Griffiths transversality condition, $\nabla$ induces an $\mathcal{O}_{X}$-linear map
\[
\theta: \Gr_{\Fil^{\bullet}}(H)\longrightarrow \Gr_{\Fil^{\bullet}}(H)\otimes \Omega_{X/S}^1,
\]
taking the class in $\Gr_{\Fil^{\bullet}}^i(H)$ of a local section $x$ in $\Fil^iH$ to the class of $\nabla(x)$ in $\Gr_{\Fil^{\bullet}}^{i-1}(H)\otimes \Omega_{X/S}^1$. In this way, $(\Gr_{\Fil^{\bullet}}(H),\theta)$ becomes a (graded) Higgs bundle. For example, if we take the \textit{trivial filtration} $\Fil_{tr}$ on $H$:
\[
\Fil_{tr}^0H=H \quad \textrm{and}\quad \Fil_{tr}^1H=0,
\]
then $\Gr_{\Fil_{tr}}(H)=H$ and the resulting Higgs field $\theta$ on $\Gr_{\Fil_{tr}}(H)$ is trivial. So $\Gr_{\Fil_{tr}}(H)$ is simply $H$, viewed as a Higgs bundle with trivial Higgs field on $X$.
\end{construction}

\begin{definition}[\cite{LSZ}] We keep the notations above. Let $f\in \mathbb Z_{>0}$.

\begin{enumerate}
\item An \emph{$f$-periodic HDRF of level zero} over $X_1/W_1$ is a diagram of the form
\begin{equation}\label{eq:HDRF-over-X1}
\xymatrix{& \textcolor{black}{(H_0,\nabla_{0})}\ar[rd]^{\Gr_{\Fil_{tr}}}&     & (H_{f-1},\nabla_{f-1})\ar[rd]^{\Gr_{\Fil_{tr}}} & \\ (E_0,0):=(E,0)\ar[ru]^{C_{1}^{-1}} & & \cdots \ar[ru]^{C_{1}^{-1}}  & & (E_{f},0)\ar@/^1.5pc/[llll]_{\phi}};
\end{equation}
Here,
\begin{itemize}
\item $(E,0)$ is a Higgs bundle with trivial Higgs field over $X_1$;
\item $C_1^{-1}$ is the inverse Cartier transform of Ogus-Vologodsky (\cite{OV});
\item $\Gr_{\Fil_{tr}}$ stands for taking the associated graded $\mathcal O_{X_1}$-module relative to the trivial filtration $\Fil_{tr}$ as recalled in Construction \ref{const:GrFil}: \textcolor{black}{in particular, if we denote by $F_{X_1}$ the absolute Frobenius on $X_1$, for each $0\leq i\leq f-1$, we have $
H_i=F_{X_1}^*E_i$, $\nabla_i$ is the canonical connection on a Frobenius pullback, and $E_{i+1}=H_i$; }
\item $\phi:(E_f,0)\stackrel{\sim}{\ra}(E,0)$ is an isomorphism of Higgs bundles on $X_1$.
\end{itemize}
We shall write this $f$-periodic HDRF by the tuple $(E,0,\Fil_{tr},\ldots,\Fil_{tr},\phi)$.

\item When $n\geq 2$, an \emph{$f$-periodic HDRF of level zero} over $X_n$ is defined inductively as the collection of the following data:
\begin{itemize}
\item an $f$-periodic HDRF of level zero $(\bar E,0,\Fil_{tr},\cdots,\Fil_{tr},\bar \phi)$ on $X_{n-1}$: in particular, we have the following diagram
\begin{equation}\label{eq:HDRF-over-Xn-1}
\xymatrix{& (\bar H_0,\nabla_0)\ar[rd]^{\Gr_{\Fil_{tr}}}&     & (\bar H_{f-1},\nabla_{f-1})\ar[rd]^{\Gr_{\Fil_{tr}}} & \\ (\bar E_0,0):=(\bar E,0)\ar[ru]^{C_{n-1}^{-1}} & & \cdots \ar[ru]^{C_{n-1}^{-1}}  & & (\bar E_{f},0)\ar@/^1.5pc/[llll]_{\bar{\phi}}}
\end{equation}
with $\bar \phi$ an isomorphism of Higgs bundles on $X_{n-1}$;
\item a Higgs bundle $(E,0)$ with trivial Higgs field on $X_{n}$, with
\[
(E,0)\otimes_{W_n}W_{n-1}\simeq (\bar E,0);
\]
because of the equality $
(\bar E_f,0)=\Gr_{\Fil_{tr}}(\bar H_{f-1},\bar{\nabla}_{f-1})$ and the isomorphism $\bar \phi$, we can apply the truncated inverse Cartier transform $C_n^{-1}$ of \cite{LSZ} to $(E,\theta)$, \textcolor{black}{from which we deduce a diagram}, i.e., the Higgs-de Rham flow initiated by $(E,0)$:
\begin{equation}\label{eq:HDRF-over-Xn}
\xymatrix{& ( H_0,\nabla_0)\ar[rd]^{\Gr_{\Fil_{tr}}}&     & (H_{f-1},\nabla_{f-1})\ar[rd]^{\Gr_{\Fil_{tr}}} & \\ (E_0,0):=(E,0)\ar[ru]^{C_{n}^{-1}} & & \cdots \ar[ru]^{C_{n}^{-1}}  & & (E_{f},0)}.
\end{equation}
\textcolor{black}{As the inverse Cartier transforms} $C_{n}^{-1}$ for various $n$ are compatible with each other, the diagrams \eqref{eq:HDRF-over-Xn-1} and \eqref{eq:HDRF-over-Xn} are compatible in the evident way; and
\item an isomorphism $\phi:(E_f,0)\stackrel{\sim}{\ra} (E,0)$ of Higgs bundles on $X_n$, compatible with $\bar \phi:(\bar E_f,0)\stackrel{\sim}{\ra} (\bar E,0)$.
\end{itemize}
Like above, we denote this $f$-periodic HDRF by the tuple
\[
(E,0,\Fil_{tr},\ldots, \Fil_{tr},\phi).
\]
\end{enumerate}
\end{definition}

Let $(E,0)$ be a Higgs bundle with trivial Higgs field on $X_{n}$, such that $(\bar E, 0):=(E,0)\otimes_{W_n}W_{n-1}$ is the initial term of an $f$-periodic HDRF
\[
(\bar E, 0, \Fil_{tr}, \ldots, \Fil_{tr}, \bar \phi)
\]
of level zero on $X_{n-1}$. \textcolor{black}{We can apply the inverse Cartier} transform $C_n^{-1}$ on $(E,0)$. Write $(H,\nabla)=C_n^{-1}(E,0)$.

\begin{lemma}\label{flat basis} Up to a canonical isomorphism, the vector bundle with integrable connection $(H,\nabla)$ is independent of the choice of the $W_{n+1}$-lifting $X_{n+1}$ of $X_{n}$. Moreover, Zariski locally, $H$ admits a flat basis with respect to $\nabla$.
\end{lemma}
\begin{proof} Consider the corresponding diagram \eqref{eq:HDRF-over-Xn} associated with the $f$-periodic HDRF $(\bar E,0,\Fil_{tr},\ldots,\Fil_{tr},\bar \phi)$ over $X_{n-1}$. In particular, $\bar \phi : (\bar E_f,0)\stackrel{\sim}{\ra}(\bar E,0)$ is an isomorphism of Higgs bundles over $X_{n-1}$, and $(\bar E_f,0)=\Gr_{\Fil_{tr}}(\bar H_{f-1},\bar{\nabla}_{f-1})$. So $\bar E_f=\bar{H}_{f-1}$. We want to compute $C_n^{-1}(E,0)$, which, according to \cite{LSZ}, means
$$
C_n^{-1}(E,0,\bar H_{f-1},\bar \nabla_{f-1},\Fil_{tr},\bar \phi).
$$
From the proof of \cite[Theorem 4.1]{LSZ}, the latter is obtained in two steps. The first step is to construct a vector bundle with integrable $p$-connection on $X_n$, given by
\begin{equation*}
 (\tilde H,\tilde \nabla)=\sT_n(E,0,\bar H_{f-1},\bar \nabla_{f-1},\Fil_{tr},\bar \phi).
\end{equation*}
We claim that $\tilde H\simeq E$. Indeed, by the second approach in the construction of $\sT_n$ in \cite{LSZ}, $\tilde H$ is the cokernel of the morphism
$$
E\longrightarrow \bar E\oplus E, \quad x\mapsto (\bar x, -px),
$$
where $\bar x$ is the reduction of $x$ mod $p^{n-1}$. The image of this morphism is equal to  the kernel of the morphism
$$
\bar E\oplus E \longrightarrow E, \quad (\bar x,y)\mapsto p\bar x+y.
$$
Thus $\tilde H\simeq E$, as claimed. Under this identification, the $p$-connection $\tilde{\nabla}$ on $E$ can be described as follows. Let $\bar{\nabla}$ be the connection on $\bar E$ induced from $\nabla_{f-1}$ on $\bar E_{f}=\bar H_{f-1}$ through the isomorphism $\bar \phi: \bar E_f\stackrel{\sim}{\ra}\bar E$. For any local section $x$ of $\tilde H\simeq E$, write $\bar x$ for its reduction modulo $p^{n-1}$. Let $\tilde{y}\in E\otimes\Omega_{X_n/W_n}^1$ be any (local) lifting of $\bar{\nabla}(\bar x)\in\bar E\otimes \Omega_{X_{n-1}/W_{n-1}}^1$. Since the difference of two such liftings lies in $p^{n-1}E\otimes \Omega_{X_n/W_n}^1$, $p\tilde{y}$ is a well-defined local section of $E\otimes \Omega_{X_n/W_n}^1$ depending only on $\bar x$. Denote the latter by $p\bar{\nabla}(\bar x)$. Then one has $\tilde{\nabla}(x)=p\bar{\nabla}(\bar x)$.

To deduce $(H,\nabla)$ from $(\tilde H, \tilde \nabla)$, assume first that $X_{n+1}$ is \emph{small}, that is, there exists an \'etale morphism of $W_{n+1}$-schemes:
\[
X_{n+1}\longrightarrow \Spec(W_{n+1}[t_1^{\pm 1},\ldots, t_d^{\pm 1}]).
\]
Let $F_{X_{n+1}}:X_{n+1}\to X_{n+1}$ be a \emph{Frobenius lifting}, i.e., a lifting of the absolute Frobenius on $X_1$ that is compatible with the Frobenius on $W_{n+1}$. Let $F_{X_n}:X_n\ra X_n$ be its reduction modulo $p^n$. By Step 1 in the proof of \cite[Proposition 4.15]{LSZ} and \cite[Formula (4.15.1)]{LSZ},
\[
H=F_{X_n}^{*}\tilde H=F_{X_n}^*E=\mathcal O_{X_n}\otimes_{F_{X_n}^{\#},\mathcal O_{X_n}}E,
\]
and for $\lambda\otimes x\in F_{X_n}^*E=H$ we have
\begin{equation}\label{formula-for-connection}
\begin{array}{rcl}\nabla(\lambda\otimes x)& =& (1\otimes x)\otimes d\lambda +\lambda \left(1\otimes \frac{dF_{X_{n+1}}}{p}\right)(1\otimes \tilde{\nabla}(x)) \\ & =& (1\otimes x)\otimes d\lambda +\lambda \left(1\otimes dF_{X_{n}}\right)(1\otimes \bar{\nabla}(\bar x)) .\end{array}
\end{equation}
Here again, since $dF_{X_n}(\Omega_{X_n/W_n}^1)\subset p\Omega_{X_n/W_n}^1$, $(1\otimes dF_{X_{n}})(1\otimes \bar{\nabla}(\bar x))$ is a well-defined section of $H\otimes\Omega_{X_n/W_n}^1$. In particular, $\nabla $ depends only on $F_{X_n}$. Moreover, let $F_{X_{n+1}}'$ be a second Frobenius lifting. Write $\nabla'$ for the connection on $H':=F_{X_{n}}^{'*}\tilde{H}$ defined similarly as above. One has a horizontal isomorphism $\sigma=\sigma_{F_{X_{n+1}},F_{X_{n+1}}'}$ between $(H,\nabla)$ and $(H',\nabla')$:
\[
\sigma : F_{X_n}^*\tilde{H}\longrightarrow F_{X_n}^{'*}\tilde{H},\quad 1\otimes x\mapsto\sum_{\underline m\in \mathbb N^d}\frac{1}{\underline m!}\left(\frac{F_{X_{n+1}}(\underline t)-F_{X_{n+1}}'(\underline t)}{p}\right)^{\underline m}\otimes \tilde{\nabla}_{\partial/\partial \underline t}^{\underline m}(x).
\]
Since $\tilde \nabla(x)=p\bar{\nabla}(\bar x)$, the series above can be rewritten as
\begin{equation}\label{eq:series-for-sigma}
\sum_{\underline m\in \mathbb N^d}\frac{(F_{X_{n}}(\underline t)-F_{X_{n}}'(\underline t))^{\underline m}}{\underline m!}\otimes \bar{\nabla}_{\partial/\partial \underline t}^{\underline m}(\bar x).
\end{equation}
Thus, $\sigma$ depends only on the reduction modulo $p^n$ of $F_{X_{n+1}}$ and $F_{X_{n+1}}'$ as well.

In the general case, we cover $X_{n+1}$ by small open subsets $X_{n+1}=\bigcup_{i\in I} U_i$. For each $i\in I$, choose a Frobenius lifting $F_i: U_i\ra U_i$. Write $\bar U_i$ (resp. $\bar{F}_i$) for the reduction modulo $p^n$ of $U_i$ (resp. of $F_i$). Then, $(H,\nabla)$ is obtained by gluing $(H_i,\nabla_i)=\bar F_i^{*}(\tilde H, \tilde{\nabla})|_{\bar U_i}$ on $\bar U_i$, $i\in I$, along the horizontal isomorphisms
\[
\sigma_{F_i, F_j} : (H_i, \nabla_i)|_{\bar U_i\cap \bar U_j}\stackrel{\sim}{\longrightarrow} (H_j, \nabla_j)|_{\bar U_i\cap \bar U_j}, \quad i, j\in I.
\]
As seen in the paragraph above, the $(H_i, \nabla_i)$ and the gluing isomorphisms $\sigma_{F_i, F_j}$ can be given using only the reduction modulo $p^n$ of $F_i$ and $F_j$, so the same holds for $(H,\nabla)$. Furthermore, a similar argument shows that, up to a canonical isomorphism, $(H,\nabla)$ does not depend on the covering $X_n=\bigcup_i \bar U_i$ and the Frobenius lifting $\bar F_i$ on $\bar U_i$. In other words, $(H,\nabla)$ depends only on the reduction modulo $p^n$ of $X_{n+1}$, as asserted in the first part of our lemma.

For the remaining part of our lemma, it is harmless to assume that $X_n$ is small. We shall do induction on $n$. If $n=1$, $\nabla$ is the canonical connection $\nabla_{can}$ of trivial $p$-curvature in the Cartier descent theorem. A flat local basis of $H=F_{X_1}^{*}E$ is given by $1\otimes e$, where $e$ is a local basis of $E$. Assume $n\geq 2$ and the truth of our assertion for $n-1$. Choose a Frobenius lifting $F_{X_n}$ on $X_n$. By the induction hypothesis, we
 may take a local basis $e$ of $E$ whose mod $p^{n-1}$ reduction $\bar e$ in $\bar E$ is mapped via the isomorphism $\bar \phi^{-1}$ to a flat local
 basis of $\bar H_{f-1}$ with respect to $\bar \nabla_{f-1}$. Then, $e$ is a flat local basis of $\tilde H_0=E$ with respect to the
 $p$-connection $\tilde \nabla_0$. By \eqref{formula-for-connection}, $1\otimes e$ is a flat local basis of $H=F_{X_{n}}^*E$ with respect to $\nabla$, and this completes the induction step.
\end{proof}

\begin{remark}\label{rem:remove-extra-assumptions} We keep the notations of Lemma \ref{flat basis} (and its proof).
\begin{enumerate}
\item \textcolor{black}{Lemma \ref{flat basis} can be viewed as a $p^n$-torsion generalization of Cartier descent (\cite[Theorem 5.1]{Katz70}). Moreover, the fact that $(H,\nabla)$ is independent of $X_{n+1}$ generalizes the fact that the inverse Cartier transform of Ogus-Vologodsky (\cite{OV}) of a Higgs bundle with trivial Higgs field specializes to the Frobenius pullback mod $p$, which clearly does not require $X_1$ to be lifted over $W_2$.} It follows that the notion of $f$-periodic HDRF of level zero on $X_n$ can be defined without any smooth $W_{n+1}$-model of $X_n$. 


\item To define periodic HDRF of level zero, one can equally remove the assumption $p\geq 3$ in the condition (*) at the beginning of this section. Indeed, in the original definition in \cite{LSZ} (which also works for non-zero levels), the only place we need the fact that $p\geq 3$ is to make sure that the series in \eqref{eq:series-for-sigma} is a finite sum (recall that, for $p$ a prime, the series $p^n/n!$ converges $p$-adically to $0$ when $p\geq 3$). However, as in the last part of the proof above, one checks that the $\bar E$ has flat basis relative to $\bar {\nabla}$ Zariski locally. Thus $\nabla^{\underline m}_{\partial/\partial \underline t}(\bar x)$ converges $p$-adically to $0$ when $|\underline m|\to \infty$. As a result, the series in \eqref{eq:series-for-sigma} is still a finite sum when $p=2$! This allows us to glue the local pieces $(H_i,\nabla_i)$ by means of the horizontal isomorphisms $\sigma_{F_i,F_j}$ when $\mathrm{Char}(k)=p=2$. 
\end{enumerate}
Therefore, to define HDRFs of level zero on $X_n$, we can drop the condition (*) at the beginning of this section. \textcolor{black}{There is a direct definition of a morphism between} two $f$-periodic HDRFs of level zero, and we obtain the category of $f$-periodic HDRFs of level zero on $X_n$, denoted by
\[
\mathrm{HDR}_{0,f}(X_n/W_n)
\]
in the sequel.
\end{remark}


The notion of periodic HDRFs of level zero is closely related with another one introduced by Katz \cite{Katz73}. To state this, suppose that $X_n/W_n$ is equipped with a Frobenius lifting $F=F_{X_n}:X_n\ra X_n$. Take $f\in \mathbb Z_{>0}$. A \emph{Frobenius-periodic vector bundle of period $f$} on $X_n$ (relative to $F$) is a pair $(E,\psi)$ consisting of a vector bundle $E$ over $X_n$, together with an isomorphism of $\mathcal{O}_{X_n}$-modules
\[
\psi:F^{*f}(E)\stackrel{\sim}{\longrightarrow} E.
\]
Such objects form a category, denoted by $\mathrm{FVect}_f(X_n,F)$.

\begin{proposition}\label{Frobenius-periodic bundles}
Assume that there exists a Frobenius lifting $F:X_{n}\to X_{n}$. Let $(E,0, \Fil_{tr},\ldots, \Fil_{tr},\phi)$ be an $f$-periodic HDRF of level zero on $X_n$, with \eqref{eq:HDRF-over-Xn} its corresponding diagram. Then, $E_f\simeq F^{*f}E$, and under this identification, $\phi$ becomes an isomorphism $\phi_F:F^{*f}E\stackrel{\sim}{\ra}E$ of $\mathcal{O}_{X_n}$-modules.  Moreover, the functor
\begin{equation}\label{eq:functor-HDRF-FVect}
\mathrm{HDR}_{0,f}(X_n/W_n)\longrightarrow \mathrm{FVect}_{f}(X_n,F),
\end{equation}
sending $(E,0,\Fil_{tr},\ldots, \Fil_{tr},\phi)$ to $(E,\phi_F)$ is an equivalence of categories.
\end{proposition}
In particular, for $n=1$, an $f$-periodic HDRF over $X_1$ is just a vector bundle $E$ on $X_1$ together with an isomorphism $\psi:\mathrm{Fr}^{*f}E\stackrel{\sim}{\ra} E$ of $\mathcal{O}_{X_1}$-modules, where $\mathrm{Fr}:X_1\ra X_1$ is the absolute Frobenius on $X_1$.
\begin{proof}
We shall prove our proposition by induction on $n$. Suppose $n=1$. Consider an $f$-periodic HDRF of level zero $(E,\Fil_{tr},\ldots,\Fil_{tr},\phi)$ on $X_1$, and its corresponding diagram \eqref{eq:HDRF-over-X1}. Since we have naturally $C_1^{-1}(\tilde E,0)\simeq (F^*\tilde E,\nabla_{can})$ for any Higgs bundle $(\tilde E,0)$ with trivial Higgs field, we get inductively a natural isomorphism $(E_i,0)\simeq F^{*i} E$, and $\phi$ becomes an isomorphism
$$
\phi_{F}:
(F_{X_1}^{*f}E,0)=\Gr_{\Fil_{tr}}(F_{X_1}^{*f}E,\nabla_{can})\stackrel{\sim}{\longrightarrow} (E,0)
$$
of Higgs bundles, or equivalently, an isomorphism of vector bundles
\[
\phi_F: F_{X}^{*f}E\stackrel{\sim}{\longrightarrow} E.
\]
Thus, the pair $(E,\phi_F)$ is a Frobenius-periodic vector bundle of period $f$. This gives the functor \eqref{eq:functor-HDRF-FVect} when $n=1$, which is an equivalence of categories.

Assume $n\geq 2$. As seen in the proof of Lemma \ref{flat basis}, $E_i\simeq F^{*i}E$ for $1\leq i\leq f$. So like above, $E_{f}\simeq F^{*f}E$ and $\phi$ becomes an isomorphism $\phi_F: F^{*f}E\stackrel{\sim}{\ra} E$ of vector bundles, giving a Frobenius-periodic vector bundle $(E,\phi_F)$ of period $f$ over $X_n$. This defines the functor \eqref{eq:functor-HDRF-FVect}. Conversely, let $(E,\psi)$ be a
 Frobenius-periodic vector bundle of period $f$ on $X_n$ and set $(\bar E,\bar \psi)=(E,\psi)\otimes \Z/p^{n-1}\Z$. By induction, we attach to $(\bar E,\bar \psi)$ to an $f$-periodic HDRF $(\bar E,0,Fil_{tr},\cdots,Fil_{tr},\bar \phi)$ over $X_{n-1}$ such that $\bar{\phi}_{\bar F}=\bar \psi$, where $\bar F$ is the reduction of $F$ modulo $p^{n-1}$. The bundle part of $C_n^{-1}(E,0)$ is canonically identified with $F_{}^{*}E$, and successively one has (set $(E_0,\theta_0)=(E,0)$)
$$
(E_i,\theta_i)=\Gr_{\Fil_{tr}}\left(C_n^{-1}(E_{i-1},\theta_{i-1})\right)\simeq (F_{X_n}^{*i}E,0), \quad 1\leq i\leq f.
$$
So, $\psi$ gives the isomorphism $\phi:E_f\stackrel{\sim}{\ra}E$ required for a periodic HDRF of level zero over $X_n$. As a result, the functor \eqref{eq:functor-HDRF-FVect} is an equivalence of categories.
\end{proof}

Later we shall need the notion of periodic HDRF of level zero over $W$. Let $\mathcal X$ be a smooth formal scheme over $W$ lifting $X/k$, and write $X_n$ the reduction modulo $p^n$ of $\mathcal X$. An \textit{$f$-periodic HDRF of level zero on $\mathcal X/W$} consists of a HDRF of level zero $(E_n,0,\Fil_{tr},\ldots,\Fil_{tr},\phi_n)$ on $X_n/W_n$ for each $n\in \mathbb Z_{>0}$, together with isomorphisms
\[
(E_{n+1},0,\Fil_{tr},\ldots, \Fil_{tr},\phi_{n+1})|_{X_n}\stackrel{\sim}{\longrightarrow} (E_n,0,\Fil_{tr},\ldots, \Fil_{tr}, \phi_n), \quad \forall \ n>0.
\]
In particular, the projective limit $\varprojlim (E_{n},0)$ gives a Higgs bundle $\mathcal E$ with trivial Higgs field on $\mathcal X$. Since the truncated inverse Cartier transforms $C_n^{-1}$ for various $n$ are compatible with each other, the projective limit
\[
C_{\mathcal X}^{-1}(E,0):=\varprojlim (C_n^{-1}(E_n,0))
\]
is a vector bundle with integrable connection on $\mathcal X$, and the compatible family $(\phi_n)_{n\geq 1}$ defines an isomorphism of Higgs bundles
\[
\phi=\varprojlim \phi_n :  (\Gr_{\Fil_{tr}}\circ C_{\mathcal X}^{-1})^f(\mathcal E,0)\stackrel{\sim}{\longrightarrow} (\mathcal E, 0).
\]
As above, we shall denote this $f$-periodic HDRF of level zero on $\mathcal X$ by the tuple $(\mathcal E, 0, \Fil_{tr},\ldots, \Fil_{tr}, \phi)$, and by $\mathrm{HDR}_{0,f}(\mathcal X/W)$ the category of such objects.

\section{A Hitchin-Simpson correspondence in positive characteristic} Let $f\geq 1$ be an integer. Let $k$ be a perfect field of positive characteristic $p$ containing $\mathbb F_{p^f}$, the finite field with $p^f$ elements. Let $X/k$ be a connected smooth scheme. In this section, we would like to, under some conditions, give an algebro-geometric parametrization of the category
\[
\Rep_{W_n(\F_{p^f})}(\pi_1(X))
\]
of continuous representations of $\pi_1(X)$ on finite free $W_n(\F_{p^f})$-modules.  For $n=1$, such a parametrization was done by Lange-Stuhler \cite{LS} and Katz \cite{Katz73} (see also \cite{Katz71}). In loc. cit., Katz also gave a parametrization for general $n$, however with the liftability assumption both on the scheme $X$ and on the absolute Frobenius morphism. This latter condition is very restrictive: for example, it is well-known that there is no Frobenius lifting over $W_n=W_n(k)$, $n\geq 2$, for any smooth projective curve of genus $\geq 2$. Notice that in the theorem below, we assume only a mild liftability condition on the scheme $X/k$.

\begin{theorem}\label{equivalence in HT zero case} Let $f\geq 1$ be an integer. Let $k$ be a perfect field of positive characteristic $p$ containing $\mathbb F_{p^f}$. Let $X$ be a connected smooth variety over $k$. Assume that $X/k$ can be lifted to a smooth scheme $X_n$ over $W_n$. Then there is an equivalence of categories between $\mathrm{HDR}_{0,f}(X_n/W_n)$ of $f$-periodic HDRFs of level zero over $X_n$ and the category $\mathrm{Rep}_{W_n(\F_{p^f})}(\pi_1(X))$ of continuous $W_{n}(\mathbb F_{p^f})$-representations of $\pi_1(X)$ on finite free $W_n(\mathbb F_{p^f})$-modules.
\end{theorem}

Let us first prove the theorem above under an extra condition.

\begin{proof}[Proof of Theorem \ref{equivalence in HT zero case} under the existence of a Frobenius lifting] Assume moreover that there exists a Frobenius lifting $F:X_n\ra X_n$. By Proposition \ref{Frobenius-periodic bundles}, it suffices to find an equivalence between the category $\mathrm{FVect}_f(X_n)$ and the category $\mathrm{Rep}_{W_n(\mathbb F_{p^f})}(\pi_1(X))$, which follows from Katz's result \cite[Proposition 4.1.1]{Katz73}. More precisely, for $(E,\psi)$ a Frobenius-periodic vector bundle of period $f$ and of rank $r$ on $X_n$, Katz defined in loc. cit. its corresponding $W_n(\mathbb F_{p^f})$-representation of $\pi_1(X)$ as follows. Firstly, for a finite \'etale morphism $f:Y\ra X_n$, $Y$ admits a unique Frobenius lifting $F_Y$ that is compatible with $F_{X_n}$. It follows that we can pull-back $(E,\psi)$ through $f$ to obtain a Frobenius periodic vector bundle of period $f$ on $Y$ relative to $F_{Y}$, written by $(E_Y, \psi_Y)$. Secondly, Katz showed in loc. cit. that one may find a suitable finite \'etale Galois morphism $f:Y\ra X_n$ such that the vector bundle $E_Y$ has a basis formed by $\psi_Y$-invariant elements. As a result, the set $E_{Y}^{\psi_Y=1}$ of $\psi_Y$-invariant elements is a free $W_n(\mathbb F_{p^f})$-module of rank $r$. Finally, the Galois group of $Y/X$ acts naturally on $E_Y^{\psi_Y=1}$, so it induces a $W_{n}(\mathbb F_{p^f})$-representation of $\pi_1(X_n)=\pi_1(X)$. Let us denote the latter intuitively by $E^{\psi=1}$. Katz showed that the correspondence $(E,\psi)\mapsto E^{\psi=1}$ gives an equivalence of categories $\mathrm{FVect}_f(X_n)\stackrel{\sim}{\ra}\mathrm{Rep}_{W_n(\mathbb F_{p^f})}(\pi_1(X))$. Consequently, we have an equivalence of categories
\begin{equation}\label{eq:local-HT-equivalence-in-level-0}
\mathrm{HDR}_{0,f}(X_n/W_n)\longrightarrow \mathrm{Rep}_{W_n(\mathbb F_{p^f})}(\pi_1(X)),
\end{equation}
which sends an $f$-periodic HDRF $(E,0, \Fil_{tr},\ldots,\Fil_{tr},\phi)$ of level zero on $X_n$ to the $W_{n}(\mathbb F_{p^f})$-representation $E^{\phi_F=1}$ of $\pi_1(X)$.
\end{proof}

A priori, the functor \eqref{eq:local-HT-equivalence-in-level-0} depends on the choice of the Frobenius lifting $F$. To obtain a proof of Theorem \ref{equivalence in HT zero case} in the general case, one has to to compare the functor \eqref{eq:local-HT-equivalence-in-level-0} relative to two Frobenius liftings.

\begin{lemma}\label{lem:existence-of-connection} We keep the assumption of Theorem \ref{equivalence in HT zero case}. Let $(E,0,\Fil_{tr},\ldots, \Fil_{tr},\phi)$ be an $f$-periodic HDRF of level $0$ on $X_n$, with the corresponding diagram \eqref{eq:HDRF-over-Xn}.
\begin{enumerate}
\item The vector bundle $E$ is endowed naturally with a connection $\nabla$, such that if there is a Frobenius lifting $F$ on some open $U\subset X_n$, so that $E_{f}|_U$ is identified naturally with $F^{*f}(E|_U)$, then the isomorphism
\[
\phi_F:F^{*f}(E|_U)\stackrel{\sim}{\longrightarrow}E|_U
\]
induced from $\phi|_U$ is horizontal. Here $F^{*f}(E|_U)$ is equipped with the connection $F^{*f}(\nabla|_U)$: for $\lambda\otimes e\in F^{*f}(E|_U)=\mathcal O_{U}\otimes_{F^{f},\mathcal O_{U}}E|_U$,
\[
F^{*f}(\nabla|_U)(\lambda\otimes e)=(1\otimes e)d\lambda +\lambda(1\otimes dF^f)(\nabla(e)).
\]
\item Assume $X_n$ small, hence equipped with \'etale coordinates $\underline t=(t_1,\ldots, t_d)$. Let $F$ and $F'$ be two Frobenius liftings on $X_n$. Then there is an isomorphism of $\mathcal{O}_{X_n}$-modules
\[
\sigma:F^{*f}E\longrightarrow F^{'*f}E,\quad 1\otimes x\mapsto\sum_{\underline m\in \mathbb N^d}\frac{(F^f(\underline t)-F^{'f}(\underline t))^{\underline m}}{\underline m!}\otimes \nabla_{\partial/\partial \underline t}^{\underline m}(x)
\]
such that $\phi_{F'}\circ \sigma=\phi_F$.
\end{enumerate}
\end{lemma}

\begin{proof} (1) Consider the diagram \eqref{eq:HDRF-over-Xn} corresponding to $(E,0,\Fil_{tr},\ldots,\Fil_{tr},\phi)$. Via the isomorphism $\phi: E_f\stackrel{\sim}{\ra}E$, the connection $\nabla_{f-1}$ on $H_{f-1}=E_f$ induces a connection $\nabla$ on $E$. If $F$ is a Frobenius lifting on some open $U\subset X_n$, from the construction the truncated Cartier transform $C_{n}^{-1}$, we find: for $\lambda\otimes e\in H_{f-1}|_U= E_{f}|_U\simeq F^{*f}E|_U$,
\[
\nabla_{f-1}(\lambda\otimes e)=(1\otimes e)d\lambda+\lambda(1\otimes dF^f)(\bar \nabla (\bar e)),
\]
where $\bar \nabla$ (resp. $\bar e$) is the reduction modulo $p^{n-1}$ of $\nabla$ (resp. of $e$). Since $\nabla(e)$ is a lifting of $\bar{\nabla}(\bar e)$ in $E\otimes \Omega_{X_n/W_n}^1$, $(1\otimes dF^f)(\bar \nabla (\bar e))=(1\otimes dF^f)(\nabla (e))$. So $F^{*f}(\nabla|_U)=\nabla_{f-1}|_U$, and the isomorphism $\phi_F$ induced from $\phi|_U$ is horizontal.

(2) According to the construction of HDRF, the isomorphism $\phi_F$ and $\phi_{F'}$ can be inserted into the commutative diagram below 
\begin{equation}\label{eq:phi-for-two-liftings}
\xymatrix{H_{f-1}\ar@{=}[d]\ar[r]^{\simeq} & F^{*f}E\ar[r]^{\phi_F} \ar[d]_{\simeq}& E \\ H_{f-1}\ar[r]^{\simeq} & F^{'*f}E\ar[ru]_{\phi_{F'}} &  },
\end{equation}
where the vertical isomorphism is
\[
F^{*f}E\longrightarrow F^{'*f}E,\quad 1\otimes x\mapsto\sum_{\underline m\in \mathbb N^d}\frac{(F^f(\underline t)-F^{'f}(\underline t))^{\underline m}}{\underline m!}\otimes \bar{\nabla}_{\partial/\partial \underline t}^{\underline m}(\bar x).
\]
Since $\nabla_{\partial/\partial \underline t}^{\underline m}(x)$ lifts $\bar{\nabla}_{\partial/\partial \underline t}^{\underline m}(\bar x)$, the series above can be rewritten as
\[
\sum_{\underline m\in \mathbb N^d}\frac{(F^f(\underline t)-F^{'f}(\underline t))^{\underline m}}{\underline m!}\otimes \nabla_{\partial/\partial \underline t}^{\underline m}(x).
\]
Hence, the vertical isomorphism in \eqref{eq:phi-for-two-liftings} is exactly the map $\sigma$ given in (3).
\end{proof}


\begin{proof}[Proof of Theorem \ref{equivalence in HT zero case} in the general case] The key point of the proof is to compare, in the local situation, the functor \eqref{eq:local-HT-equivalence-in-level-0} relative to two different Frobenius liftings. First of all, let us assume that $X_n$ is small, thus equipped with \'etale coordinates $\underline{t}=(t_1,\ldots,t_d)$. Let $F,F'$ be two Frobenius liftings on $X_n$. Let $(E,\Fil_{tr},\ldots,\Fil_{tr},\phi)$ be an $f$-periodic HDRF of level zero on $X_n$. Relative to the lifting $F$ (resp. the lifting $F'$), the isomorphism $\phi:E_{f}\stackrel{\sim}{\ra} E$ becomes an isomorphism of $\mathcal{O}_{X_n}$-modules
\[
\phi_F: F^{*f}E\stackrel{\sim}{\longrightarrow}E, \quad (\textrm{resp. }\phi_{F'}:F^{'*f}E\stackrel{\sim}{\longrightarrow}E),
\]
and the isomorphisms $\phi_F$ and $\phi_{F'}$ are related by the $\mathcal O_{X_n}$-linear isomorphism $\sigma$ in Lemma \ref{lem:existence-of-connection} (3).
We claim that $E^{\phi_F=1}=E^{\phi_{F'}=1}$ as representations of $\pi_1(X)$. Indeed, observe first that, for an \'etale local section $x\in E^{\phi_F=1}$, since the connection $\nabla$ on $E$ is compatible with $\phi_F$, one has
\[
\nabla(x)=\nabla(\phi_F(x))=(\phi_F\otimes dF)(\nabla(x))=\ldots=(\phi_F\otimes dF)^n(\nabla(x)).
\]
But $dF(\Omega_{X_n/W_n}^1)\subset p\Omega_{X_n/W_n}^1$, so $(dF)^n(\Omega_{X_n/Wn}^1)=0$ and $\nabla(x)=0$. Thus,
\[
\sigma(x)=\sum_{\underline m\in \mathbb N^d}\frac{(F^f(\underline t)-F^{'f}(\underline t))^{\underline m}}{\underline m!}\otimes \nabla_{\partial/\partial \underline t}^{\underline m}(x)=1\otimes x\in F^{'*f}E.
\]
Furthermore, from the commutativity of \eqref{eq:phi-for-two-liftings} and the fact that $\phi_F(x)=x$, we find $\phi_{F'}(x)=x$. As a result, the set of $\phi_F$-invariant \'etale local sections of $E$ coincides with the set of $\phi_{F'}$-invariant \'etale local sections. Therefore the $W_n(\mathbb F_{p^f})$-representations $E^{\phi_F=1}$ and $E^{\phi_{F'}=1}$ of $\pi_1(X)$ are the same.

For a general smooth $W_n$-scheme $X_n$, it can be covered by its small open subsets $X_n=\bigcup_i U_i$. Choose a Frobenius lifting $F_i$ on $U_i$ for each $i$. Let $(E,\Fil_{tr}, \ldots, \Fil_{tr},\phi)$ be an $f$-periodic HDRF of level zero on $X_n$. It gives rise to a $W_n(\mathbb F_{p^f})$-representation of $\pi_1(U_i)$, or equivalently, a finite \'etale $W_{n}(\mathbb F_{p^f})$-module scheme $Y_i$ over $U_i$. By the analysis in the first paragraph of the proof, the $Y_i$ and $Y_j$ are canonically isomorphic over $U_i\cap U_j$. Thus we can glue the $Y_i$'s to a finite \'etale $W_{n}(\mathbb F_{p^f})$-module scheme $Y$ over $X_n$. In particular, we obtain a $W_n(\mathbb F_{p^f})$-representation of $\pi_1(X_n)=\pi_1(X)$, written intuitively by $E^{\phi=1}$. In this way, we obtain a functor
\begin{equation}\label{eq:functor-for-HT-equivalence}
\mathrm{HDR}_{0,f}(X_n)\longrightarrow \mathrm{Rep}_{W_n(\mathbb F_{p^f})}(\pi_1(X)), \quad (E,\Fil_{tr},\ldots, \Fil_{tr}, \phi)\mapsto E^{\phi=1}.
\end{equation}
If $X_n$ is endowed with a Frobenius lifting $F_{X_n}$, the functor is the same as the local one \eqref{eq:local-HT-equivalence-in-level-0}.

It remains to show that the functor \eqref{eq:functor-for-HT-equivalence} is an equivalence of categories. We shall do it by defining a quasi-inverse, whose construction is already contained in \cite{Katz73}. Let $\rho:\pi_1(X_n)\ra \mathrm{Aut}_{W_n(\mathbb F_{p^f})}(M)$ be a representation of $\pi_1(X)=\pi_1(X_n)$, with $M$ a finite free $W_n(\mathbb F_{p^f})$-module. Let $f:Y_n\ra X_n$ be a finite \'etale Galois cover such that $\rho$ factors through the surjection $\pi_1(X)\ra \mathrm{Aut}(Y_n/X_n)^{\rm op}$. In particular, $\mathrm{Aut}(Y_n/X_n)$ has a right $W_{n}(\mathbb F_{p^f})$-linear action on $M$. Consider the quotient $E_{\rho}$ of $M\otimes_{W_n(\mathbb F_{p^f})}\mathcal O_{Y_n }$ by the following right action of $\mathrm{Aut}(Y_n/X_n)$: for all $\gamma\in \mathrm{Aut}(Y_n/X_n)$ and $m\otimes a\in M\otimes_{W_n(\mathbb F_{p^f})}\mathcal O_{Y_n}$,
\[
(m\otimes a)\cdot \gamma:= (m\cdot \gamma)\otimes (\gamma^{*}a).
\]
So $E_{\rho}$ is a vector bundle on $X_n$. Moreover, there is a natural connection $\nabla$ on $E$ induced from the differential $d:\mathcal O_{Y_n}\ra \Omega_{Y_n/W_n}^1$. Let $U\subset X_n$ be a small open of $X_n$ and $V:=f^{-1}(U)$. Let $F_U$ be a Frobenius lifting. As $V\ra U$ is finite \'etale, $V$ is endowed a unique Frobenius lifting $F_V$ compatible with $F_U$. By the uniqueness, $F_V$ commutes with every $U$-automorphism of $V$. As a result, the map
\[
1\otimes F_V^f: M\otimes_{W(\mathbb F_{p^f})}\mathcal O_{V}\longrightarrow M\otimes_{W(\mathbb F_{p^f})}\mathcal O_{V}
\]
descends to a $F_U^f$-semilinear morphism $\psi_U$ on $E_U=E_{\rho}|_U$, such that the induced morphism  $F_U^{*f}E_U\ra E_U$ is a horizontal isomorphism. In particular, we get a Frobenius periodic vector bundle $(E_U,\psi_U)$ relative to $F_U$, whence an $f$-periodic HDRF $(E_U,\Fil_{tr},\ldots, \Fil_{tr},\phi_U)$ on $U$. If $F_U'$ is a second Frobenius lifting on $U$, and let $\psi_U':F_{U}'E\stackrel{\sim}{\ra}E_U$ be the isomorphism resulting from $F_U'$. One checks that $\psi_U$ and $\psi_U'$ are related by the Taylor series defining by the connection $\nabla$. As a result, the $f$-periodic HDRF of level zero resulting from Frobenius periodic vector bundle $(E_U,\psi_U')$ is the same as $(E_U,\Fil_{tr},\ldots,\Fil_{tr}, \phi_U)$, i.e., the latter does not depend on the choice of $F_U$. Consequently, one can glue these local HDRFs to get an $f$-periodic HDRF of level zero on $X_n$, written by $(E_{\rho},\Fil_{tr},\ldots,\Fil_{tr},\phi_{\rho})$. The required quasi-inverse of \eqref{eq:functor-for-HT-equivalence} is given by
\[
\mathrm{Rep}_{W_n(\mathbb F_{p^f})}(\pi_1(X))\longrightarrow \mathrm{HDR}_{0,f}(X_n), \quad \rho\mapsto (E_{\rho},\Fil_{tr},\ldots, \Fil_{tr},\phi_{\rho}).
\]
This completes the proof of our theorem.
\end{proof}

The functors \eqref{eq:functor-for-HT-equivalence} for various $n$ are compatible with each other. In other words, we have the following commutative diagram of categories:
\[
\xymatrix{\mathrm{HDR}_{0,f}(X_n)\ar[r]^<<<<<<{\eqref{eq:functor-for-HT-equivalence}} \ar[d]_{(-)|_{X_{n-1}} }& \mathrm{Rep}_{W_n(\mathbb F_{p^f})}(\pi_1(X)) \ar[d]^{\mathrm{mod}\ p^{n-1}}\\\mathrm{HDR}_{0,f}(X_{n-1})\ar[r]^<<<<<{\eqref{eq:functor-for-HT-equivalence}}  & \mathrm{Rep}_{W_{n-1}(\mathbb F_{p^f})}(\pi_1(X))  }
\]
Letting $n$ tend to infinity, we obtain from Theorem \ref{equivalence in HT zero case} a Hitchin-Simpson correspondence for the category $\mathrm{Rep}_{W(\mathbb F_{p^f})}(\pi_1(X))$ of \textcolor{black}{continuous} $W(\F_{p^f})$-representation of $\pi_1(X)$ on finite free $W(\F_{p^f})$-modules.

\begin{corollary}\label{cor:HT-equivalence-in-formal-case}
Let $X$ be a connected smooth variety over $k$. Assume that $X/k$ can be lifted to a smooth formal scheme $\mathcal X$ over $W=W(k)$. Then, for each integer $f>0$, there is an equivalence of categories
\[
\mathrm{HDR}_{0,f}(\mathcal X/W)\stackrel{\sim}{\longrightarrow}
\mathrm{Rep}_{W(\F_{p^f})}(\pi_1(X)), \quad (\mathcal E,0, \Fil_{tr},\ldots, \Fil_{tr},\phi)\mapsto \mathcal E^{\phi=1}.
\]
\end{corollary}

\begin{remark} Theorem \ref{equivalence in HT zero case} and Corollary \ref{cor:HT-equivalence-in-formal-case} generalize the result \cite[Proposition 4.1.1]{Katz73} of Katz in the sense that we remove the assumption on the existence of Frobenius lifting. Note that, a rational version of Corollary \ref{cor:HT-equivalence-in-formal-case} was obtained by Crew in terms of \textit{unit-root $F$-isocrystals} (\cite[Theorem 2.1]{Crew}).
\end{remark}

In the remaining part of this section, we compare \eqref{eq:functor-for-HT-equivalence} with a construction in \cite{LSZ}. Let $\mathcal X/W$ be a proper smooth scheme with connected fibers, and set $X_n:=\mathcal X\otimes_W W_n$, $n\in \mathbb Z_{\geq 1}$. So $X:=X_1$ is the closed fiber of $\mathcal X/W$. Let $K$ be the fraction field of $W$. Write $\mathcal X_K$ for the generic fiber of $\mathcal X$.  Let $(E,0,\Fil_{tr},\cdots,\Fil_{tr},\phi)$ be an $f$-periodic HDRF of level zero over $X_n$, with
\[
\rho_1: \pi_1(X)\longrightarrow \mathrm{GL}_{r}(W_n(\mathbb F_{p^f}))
\]
the representation of $\pi_1(X)$ given by Theorem \ref{equivalence in HT zero case}. On the other hand,  by \cite[Theorem 5.3]{LSZ} and \cite[Theorem 2.6]{Fa1}, there exists another $W_n(\F_{p^f})$-representation
\[
\rho_2: \pi_1(\mathcal X_K) \longrightarrow \mathrm{GL}_r(W_n(\mathbb F_{p^f})).
\]
of $\pi_1(\mathcal X_K)$  attached to $(E, 0, \Fil_{tr},\ldots, \Fil_{tr},\phi)$. We want to compare $\rho_1$ with $\rho_2$.

\begin{proposition}\label{factor through specialization map Part 1}We keep the notations above. Then, the representation $\rho_2$ factors through the specialization morphism $\mathrm{sp}:\pi_1(\mathcal X_K)\twoheadrightarrow \pi_1(X)$, and the resulting $W_n(\mathbb F_{p^f})$-representation of $\pi_1(X)$ is isomorphic to $\rho_1$.
\end{proposition}
\begin{proof} According to \cite[Theorem 5.3]{LSZ}, we can associate to the $f$-periodic HDRF $(E,\Fil_{tr},\ldots,\Fil_{tr},\phi)$ a Fontaine module $H:=(H,\nabla_H, \Fil_{tr}, \Phi)$ with $W_n(\mathbb F_{p^f})$-endomorphism structure $
\iota:W_n(\mathbb F_{p^f})\rightarrow \mathrm{End}(H,\nabla_H, \Fil_{tr}, \Phi)$. The proof is to examine carefully the construction of $\rho_2$, which is obtained by applying to $H$ the (covariant version of the) functor $\mathbf D$ of Faltings given in \cite{Fa1}. 
As in the proof of \cite[Theorem 5.3]{LSZ}, we view an $f$-periodic HDRF of level zero as an $1$-periodic HDRF of level zero endowed with $W_n(\mathbb F_{p^f})$-endomorphism structure, and similarly a $W_{n}(\mathbb F_{p^f})$-representation of a profinite group $\Gamma$ as a $W_n(\mathbb F_p)$-representation of $\Gamma$ with $W_n(\mathbb F_{p^f})$-endomorphism structure. Consequently, to prove our proposition, we reduce to the case where $f=1$. Thus $(E,0,\Fil_{tr},\phi)$ is an $1$-periodic HDRF on $X_n$, with $(H,\nabla_H,\Fil_{tr} ,\Phi)$ its associated Fontaine module. In concrete terms, $(E,0)$ is a Higgs vector bundle with trivial Higgs field on $X_n$, $
(H,\nabla_H)=C_{n}^{-1}(E,0)$,
\[
\phi: \Gr_{\Fil_{tr}}(H,\nabla_H)=(H,0)\longrightarrow (E,0)
\]
is an isomorphism of Higgs bundles, and $\Phi=C_n^{-1}(\phi)$. 
Before starting the proof which is divided into several steps, recall that the vector bundle $E$ on $X_n$ is endowed with a connection $\nabla_E$ by Lemma \ref{lem:existence-of-connection}.
\vskip 2mm

 {\itshape Step 1:} Let $\mathcal U\subset \mathcal X$ be an affine open subset of $\mathcal X$ with non-empty closed fiber, which is equipped with an \'etale morphism over $W$:
\[
\mathcal U\longrightarrow \Spec(W[t_1^{\pm 1},\ldots, t_d^{\pm 1}]).
\]
Let $\hat{\mathcal U}=\Spf(R)$ be the $p$-adic completion of $\mathcal U$, and let $F=F_{\mathcal U}$ be a Frobenius lifting on $\hat{\mathcal U}$ which is \'etale in characteristic in characteristic $0$ (for example, we may take $F$ such that $F^{\#}(t_i)=t_i^p$ for $1\leq i\leq d$). Let $U_n$ (resp. $F_n=F_{U_n}$) be the reduction modulo $p^n$ of $\mathcal U$ (resp. of $F$). Set
\[
(E_{\mathcal U},\nabla_{E_{\mathcal U}},\phi_{\mathcal U}):=(E,\nabla_E,\phi)|_{\mathcal U},\quad \textrm{and}\quad
(H_{\mathcal U},\nabla_{H_{\mathcal U}},\Phi_{\mathcal U}):=(H,\nabla_H,\Phi)|_{\mathcal U}.
\]
Using the Frobenius lifting $F$, we get a natural isomorphism $H_{\mathcal U}\stackrel{\simeq}{\ra} F^*E_{\mathcal U}$, so that $\phi_{\mathcal U}$ becomes an $F^{\#}$-semilinear endomorphism
\[
\phi_{\mathcal U,F}:E_{\mathcal U}\longrightarrow E_{\mathcal U}.
\]
Similarly, the bundle part of $C_{n}^{-1}(H_{\mathcal U},0)$ is identified with $F^*H_{\mathcal U}$, and then $\Phi_{\mathcal U}$ becomes an $F^{\#}$-semilinear endomorphism $\Phi_{\mathcal U,F}$ on $H_{\mathcal U}$. Under the identification $H_{\mathcal U}\simeq F^*E_{\mathcal U}$, the latter can be further identified with
\[
F^{\#}\otimes \phi_{\mathcal U,F}: F^*E_{\mathcal U}\longrightarrow F^*E_{\mathcal U}.
\]

Shrinking $\mathcal U$ if needed, we assume that $H_{\mathcal U}$ admits a flat basis $(m_i)_{1\leq i\leq r}$ relative to $\nabla_{H_{\mathcal U}}$ (Lemma \ref{flat basis}). So we get an invertible matrix $A\in \mathrm{GL}_r(R/p^n)$ such that
\[
\Phi_{\mathcal U,F}(1\otimes \mathbf m)=\mathbf m\cdot A,
\]
with $\mathbf m:=(m_1,\ldots, m_r)$.

 \vskip 2mm

 {\itshape Step 2:} Let $\xi\in \mathcal X$ be the generic point of the closed fiber of $\mathcal X/W$, so that $\mathcal{O}_{\mathcal X,\xi}$ is a DVR with maximal ideal $(p)\subset \mathcal O_{\mathcal X,\xi}$. Fix an algebraic closure $\Omega$ of the quotient field of $\hat{\mathcal O}_{\mathcal X,\xi}$, the $p$-adic completion of $\mathcal{O}_{\mathcal X,\xi}$. Since $\xi\in \mathcal U$, we have a natural inclusion $\mathcal O_{\mathcal X}(\mathcal U)\hookrightarrow \mathcal O_{\mathcal X,\xi}$. On passing to $p$-adic completions, we find
\[
R=\mathcal O_{\mathcal X}(\mathcal U)^{\hat{}}\subset \hat{\mathcal O}_{\mathcal X,\xi}\subset \Omega.
\]
Let $\bar R$ be the union of all finite extensions of $R$ contained in $\Omega$ that are \'etale in characteristic zero, with $\hat{\bar R}$ its $p$-adic completion. Let $R^{ur}\subset \bar R$ be the maximal subextension of $\bar R/R$ that is \'etale everywhere. Recall that the period ring $B^+(R)$ used in \cite{Fa1} (and in \cite{LSZ}) is endowed with a Frobenius $\sigma:B^+(R)\ra B^+(R)$. As $F$ is \'etale in characteristic $0$, there is a morphism
 \[
 \iota=\iota_F:R\longrightarrow B^+(R)
 \]
 compatible with the Frobenius morphisms (\cite[page 36]{Fa1}). Furthermore, the Frobenius lifting $F$ extends uniquely to a Frobenius lifting on $R^{ur}$, and thus $\iota $ above extends to a morphism $R^{ur}\ra B^{+}(R)$, still denoted by $\iota$ in the following, which is Frobenius-compatible and makes commutative the following diagram
\[
\xymatrix{B^+(R)\ar[r]^{\theta}& \hat{\bar R} \\ R^{ur} \ar@{^(->}[u]^{\iota}\ar@{^(->}[ru]_{\textrm{natural inclusion}} & }.
\]
Here $\theta:B^+(R)\ra \hat{\bar R}$ is the natural morphism of rings defined by Fontaine.

Consider
\[
B^+(R)\otimes_{R} H_{\mathcal U}=B^+(R)/p^n\otimes_{R/p^n} H_{\mathcal U}
\]
which is endowed with the Frobenius $\sigma\otimes \Phi_{\mathcal U,F}$. The Frobenius-invariants
\begin{equation*}\label{eq:Frobenius-invariants}
\mathbf D(H_{\mathcal U})_F:=\left(B^+(R)/p^n\otimes_{R/p^n}H_{\mathcal U}\right)^{\sigma\otimes\Phi_{\mathcal U,F}=1}
\end{equation*}
is of form $(1\otimes \mathbf m)\cdot \mathbf x$, with $\mathbf x$ a column vector with entries in $B^+(R)/p^n$ satisfying
\[
A\cdot \sigma(\mathbf x)=\mathbf x.
\]
To describe $\mathbf{D}(H_{\mathcal U})_F$, one needs to solve the equation below in $B^+(R)/p^n$:
\begin{equation}\label{frobenius invariant}
\sigma(\mathbf x)=B\cdot \mathbf x,
\end{equation}
with $B=A^{-1}\in \mathrm{GL}_r(R/p^n)$, viewed as an element of $\mathrm{GL}_r(B^+(R)/p^n)$ via $\iota$. As usual one first solves the equation \eqref{frobenius invariant} modulo $p$ and then lifts the solutions to higher $p$-power. In the first step, one reduces to solving
\begin{equation}\label{Kummer equation}
\mathbf x^p=B_1\cdot \mathbf x,\quad \textrm{with}\ B_1=B \ \mathrm{mod}\  p.
\end{equation}
Since $B_1$ is invertible with entries in $R/p\stackrel{\iota}{\hookrightarrow} B^+(R)/p$, \eqref{Kummer equation} is an Artin-Schreier equation in characteristic $p$. So the entries of all the solutions $\mathbf x_1$ of \eqref{Kummer equation} lie in $R^{ur}/p\stackrel{\iota}{\hookrightarrow}B^+(R)/p$. Assume that we have obtained all the solutions of \eqref{frobenius invariant} modulo $p^{i}$ with $1\leq i\leq n-1$, and that the entries of them are contained in $R^{ur}/p^{i}\stackrel{\iota}{\hookrightarrow}B^+(R)/p^{i}$. Let $\mathbf x_i$ be one such solution. In the second step, one takes an arbitrary lifting $\mathbf x'$ of $\mathbf x_{i}$ with entries contained in $R^{ur}/p^{i+1}\stackrel{\iota}{\hookrightarrow}B^+(R)/p^{i+1}$, and looks for solutions of the from $\mathbf x'+p^{i}\mathbf y \mod p^{i+1}$ of the equation (\ref{frobenius invariant}) modulo $p^{i+1}$, which amounts to solving the next equation in $B^+(R)/p$:
\begin{equation}\label{Artin-Schreier}
\mathbf y^p-B_1\cdot \mathbf y=\mathbf z \ \mathrm{mod}\ p,
\end{equation}
where $\mathbf z$ is a vector with entires in $R^{ur}/p^{i+1}$ with $p^i\mathbf z=B\mathbf x'-F(\mathbf x')$: note that $\mathbf x_{i}=\mathbf x' \ \mathrm{mod}\ p^{i}$ is a solution of \eqref{frobenius invariant} modulo $p^i$, the entries of $B\mathbf x'-F(\mathbf x')$ is divisible by $p^i$ in $R^{ur}/p^{i+1}$. Since $B_1$ is invertible, \eqref{Artin-Schreier} is again an Artin-Schreier equation in characteristic $p$, whose solutions have entries contained in $R^{ur}/p$. It follows that all the solutions of \eqref{frobenius invariant} in $B^+(R)/p^{i+1}$ have entries in $R^{ur}/p^{i+1}$. Thus, inductively we see that  the entries of the solutions of \eqref{frobenius invariant} are contained entirely in $R^{ur}/p^n$, and the induced map
\begin{equation}\label{eq:Phi=1}
H_{\mathcal U}^{\Phi_{\mathcal U,F}=1}:=
(R^{ur}\otimes_RH_{\mathcal U})^{F\otimes\Phi_{\mathcal U,F}=1}\longrightarrow \mathbf{D}(H_{\mathcal U})_F
\end{equation}
is bijective. In particular, $\mathbf D(H_{\mathcal U})_F$ is a finite free $\mathbb Z/p^n\mathbb Z$-module.
 \vskip 2mm

{\itshape Step 3:} Let $
\Gamma:=\mathrm{Gal}(\bar R/R)$, and $\Gamma^{ur}:=\mathrm{Gal}(R^{ur}/R)$. So $\Gamma$ can be identified with the fundamental group $\pi_1(\hat{\mathcal U}_K, \bar \eta)$ of the generic fiber $\hat{\mathcal U}_K$ of $\hat{\mathcal U}$ relative to the geometric base point $\bar \eta$ defined by the inclusion $R\hookrightarrow \Omega$, and $\Gamma^{ur}$ is naturally a quotient of $\Gamma$. As in \cite[page 37]{Fa1}, $\Gamma$ acts on $B^+(R)\otimes_R H_{\mathcal U}$, from where $\mathbf D(H_{\mathcal U})_F$ inherits an action of $\Gamma$. We caution the readers that this action of $\Gamma$ on $B^+(R)\otimes_R H_{\mathcal U}$ is not the obvious action of $\Gamma$ on the first factor: in fact the latter is not well-defined as the morphism $\iota:R\ra B^+(R)$ is not $\Gamma$-equivariant.

We claim that the bijection \eqref{eq:Phi=1} is $\Gamma$-equivariant. To check this, observe that the natural map $\theta:B^+(R)\ra \hat{\bar R}$ induces a morphism
\[
B^+(R)\otimes_R H_{\mathcal U}\longrightarrow \hat{\bar R}\otimes_R H_{\mathcal U}
\]
which is $\Gamma$-equivariant: here $\Gamma$ acts on $\hat{\bar R}\otimes_R H_{\mathcal U}$ through its natural action on the first factor. Consider the following commutative diagram
\[
\xymatrix{R^{ur}\otimes H_{\mathcal U} \ar[r] & B^+(R)\otimes_R H_{\mathcal U} \ar[r] & \hat{\bar R}\otimes_R H_{\mathcal U} \\ H_{\mathcal U}^{\Phi_{\mathcal U}=1}\ar[r]^{\eqref{eq:Phi=1}}_{\simeq} \ar@{^(->}[u] & \mathbf D(H_{\mathcal U})\ar@{^(->}[u] \ar[ru]_{\alpha}& }.
\]
The composition of the upper horizontal maps is $
R^{ur}\otimes H_{\mathcal U}\rightarrow  \hat{\bar R}\otimes_R H_{\mathcal U}$,
which is injective and $\Gamma$-equivariant. Since the lower horizontal map \eqref{eq:Phi=1} is bijective, the map $\alpha$ is injective (and $\Gamma$-equivariant). Therefore, \eqref{eq:Phi=1} is $\Gamma$-equivariant. In particular, the action of $\Gamma$ on $\mathbf D(H_{\mathcal U})$ factors through the quotient $\Gamma\twoheadrightarrow \Gamma^{ur}$.

On the other hand, the natural map $E_{\mathcal U}\ra H_{\mathcal U}\simeq F^*E_{\mathcal U}$ induces an isomorphism
\[
E_{\mathcal U}^{\phi_{\mathcal U,F}=1}\stackrel{\sim}{\longrightarrow} H_{\mathcal U}^{\Phi_{\mathcal U,F}=1}
\]
of $W_{n}(\mathbb F_p)$-representations of $\Gamma^{ur}$. As a result, we obtain a natural isomorphism
\begin{equation}\label{eq:local-iso}
E_{\mathcal U}^{\phi_{\mathcal U,F}=1}\stackrel{\sim}{\longrightarrow}\mathbf D(H_{\mathcal U})_F
\end{equation}
of $W_n(\mathbb F_p)$-representations of $\Gamma^{ur}$. As $H_{\mathcal U}\simeq F^*E_{\mathcal U}$, we have also
\[
\mathbf D(H_{\mathcal U})_F=(B^+(R)\otimes_{\iota, R}H_{\mathcal U})^{\sigma\otimes \Phi_{\mathcal U,F}=1}\simeq (B^+(R)\otimes_{\iota\circ F,R}E_{\mathcal U})^{\sigma\otimes \phi_{\mathcal U,F}=1}.
\]

\vskip 2mm

{\itshape Step 4:} Assume that $\mathcal U$ is endowed with a second Frobenius lifting $F'$. Let $\phi_{\mathcal U,F'}$ be the corresponding $F^{'\#}$-semilinear endomorphism on $E_{\mathcal U}$, and $\iota'=\iota_{F'}:R\ra B^+(R)$ the Frobenius-compatible map constructed from $F'$. We have seen in the proof of Theorem \ref{equivalence in HT zero case} that
\[
E_{\mathcal U}^{\phi_{\mathcal U,F}=1}=E_{\mathcal U}^{\phi_{\mathcal U,F'}=1}\subset E_{\mathcal U}^{\nabla_{E_{\mathcal U}}},
\]
from where we deduce a commutative diagram
\[
\xymatrix{E_{\mathcal U}^{\phi_{\mathcal U,F}=1}\ar@{^(->}[r]\ar@{=}[d] & E_{\mathcal U}\ar[r] & F^*E_{\mathcal U}\ar[r]\ar[d]^{\tau} & B^+(R)\otimes_{\iota\circ F, R}E_{\mathcal U}\ar[d]^{\tau} \\ E_{\mathcal U}^{\phi_{\mathcal U,F'}=1}\ar@{^(->}[r] & E_{\mathcal U}\ar[r] & {F'}^*E_{\mathcal U}\ar[r] & B^+(R)\otimes_{\iota'\circ F',R}E_{\mathcal U}  },
\]
where $\tau$ is the usual isomorphism defined using the connection $\nabla_{E_{\mathcal U}}$ on $E_{\mathcal U}$. Taking Frobenius-invariants we obtain the following commutative diagram
\[
\xymatrix{E_{\mathcal U}^{\phi_{\mathcal U,F}=1}\ar[r]^<<<<<{\eqref{eq:local-iso}}\ar@{=}[d] & (B^+(R)\otimes_{\iota\circ F, R}E_{\mathcal U})^{\sigma\otimes \phi_{\mathcal U,F}=1}\simeq \mathbf D(H_{\mathcal U})_F\ar[d]^{\tau} \\ E_{\mathcal U}^{\phi_{\mathcal U,F'}=1}\ar[r]^<<<<<{\eqref{eq:local-iso}} & (B^+(R)\otimes_{\iota'\circ F',R}E_{\mathcal U} )^{\sigma\otimes \phi_{\mathcal U,F'}=1}\simeq \mathbf D(H_{\mathcal U})_{F'} }.
\]
In particular, up to a canonical isomorphism, the $(\mathbb Z/p^n\mathbb Z)[\Gamma]$-module $\mathbf D(H_{\mathcal U})_F$ does not depend on the choice of the Frobenius lifting $F$.

\vskip 2mm

 {\itshape Step 5:} Finally, let $\mathcal X=\bigcup_i\mathcal U_i$ be a cover of $\mathcal X$ by its small affine open subsets. For each $i$, let $F_i$ be a Frobenius lifting on $\hat{\mathcal U}_i$, the $p$-adic completion of $\mathcal U_i$. For each $i$, let $\Gamma_i$ be the fundamental group of $\hat{\mathcal U}_{i,K}$ relative to the geometric base point given by the inclusion $\mathcal O_{\mathcal X}(\mathcal U_i)^{\hat{}}\hookrightarrow \Omega$. Using the action of $\Gamma_i$ on the finite $\Z/p^n\Z$-module $\mathbf D(H_{\mathcal U_i})_{F_i}$ in Step 3, the latter corresponds naturally a finite \'etale cover of $\hat{\mathcal U}_{i,K}$, and thus (by taking normalization) a finite morphism of formal schemes over $W$
\[
\mathsf Z_{\mathcal U_i,F_i}\longrightarrow \hat{\mathcal U_i}
\]
that is \'etale in characteristic $0$. Moreover, according to Step 4, these finite morphisms for various $i$ can be glued naturally to a finite morphism $\mathsf Z\ra \hat{\mathcal X}$, which, by properness of $\mathcal X/W$, comes from a finite morphism
\[
\mathcal Z\longrightarrow \mathcal X
\]
that is \'etale in characteristic $0$. By definition, $\rho_2$ is the associated $\mathbb Z/p^n\mathbb Z$-representation of $\pi_1(\mathcal X_K)$ of the finite \'etale cover $\mathcal Z_K\ra \mathcal X_K$. On the other hand,
let $\mathcal Y\ra \mathcal X$ be the finite \'etale cover corresponding to $\rho_1$. The $\Gamma$-equivariant isomorphism \eqref{eq:local-iso} gives rise to an isomorphism over $\mathcal U_i$
\[
\mathcal Y\times_{\mathcal X}\mathcal U_i\stackrel{\sim}{\longrightarrow} \mathcal Z\times_{\mathcal X} \mathcal U_i.
\]
By Step 4, these local isomorphisms for various $i$ glue to an $\mathcal X$-isomorphism
\[
\mathcal Y\stackrel{\sim}{\longrightarrow}\mathcal Z.
\]
Consequently, $\rho_2$ factors through the quotient $sp:\pi_1(\mathcal X_K)\twoheadrightarrow \pi_1(X)$ and the resulting representation of $\pi_1(X)$ is isomorphic to $\rho_1$.
\end{proof}

\begin{corollary}\label{cor:byproduct}
Let $\mathcal X$ be a smooth proper scheme over $W$, with $X$ its special fiber and $\mathcal X_K$ its generic fiber. Let $\mathbb L_K$ be a $\Z/p^n$-\'etale local system on $\mathcal X_K$. Then, the following statements are equivalent:
\begin{enumerate}
\item $\mathbb L_K$ is crystalline of Hodge-Tate weight $0$;
\item $\mathbb L_K$ is the generic fiber of a $\Z/p^n$-\'etale local system on $\mathcal X$. In other words, the corresponding $\Z/p^n$-representation of $\pi_1(\mathcal X_K)$ factors through the quotient $\pi_1(\mathcal X_K)\twoheadrightarrow \pi_1(X)$.
\end{enumerate}
In particular, the functor ${\bf D}$ in \cite[Theorem 2.6]{Fa1} induces an equivalence of the category of strict $p^n$-torsion Fontaine modules of Hodge-Tate weight
zero and the category of crystalline $W_n(\F_p)$-representations of $\pi_1(X)$.
\end{corollary}
\begin{proof} This combines Theorem \ref{equivalence in HT zero case} and Proposition \ref{factor through specialization map Part 1}.
\end{proof}

\section{Deformation of HDRFs of level zero}

Let $k$ be a perfect field of characteristic $p$. Let $W:=W(k)$ be the ring of Witt vectors with coefficients in $k$, with $W_m$ its reduction modulo $p^m$ for every $m\in \mathbb N$.  Let $n\geq 1$ be an integer. Let $X$ be a connected smooth $k$-scheme, equipped with a geometric point $\bar x$. Suppose that $X/k$ can be lifted to a smooth $W_{n+1}$-scheme $X_{n+1}$ and write $X_n=X_{n+1}\otimes_{W_{n+1}}W_n$. Let
\[
\rho:\pi_1(X,\bar x)\longrightarrow \mathbf{GL}_d(\Z/p^n\Z)
\]
be a continuous $\Z/p^n\Z$-representation of $\pi_1(X,\bar x)$, with $(E,\Fil_{tr},\phi)$ the corresponding $1$-periodic HDRF of level zero on $X_n$ given by Theorem \ref{equivalence in HT zero case}. In the following, we consider a part of the deformation theory of $\rho$, or equivalently of the level-zero HDRF $(E,\Fil_{tr},\phi)$, and discuss its relation with the deformation theory of the vector bundle $E$. We also draw the attention of the readers to a very recent work of Krishnamoorthy-Yang-Zuo (\cite{KYZ}), where a deformation theory for HDRFs of general levels is developed in a different context.

\begin{remark}\label{rem:deformation-of-repns} Let us briefly recall some deformation theory of the representations of profinite groups \textcolor{black}{(see \cite[\S~1]{Ma} for a detailed discussion)}. Let $(R,\mathfrak m)\ra (R_0,\mathfrak m_0)$ be a surjective morphism of local Artinian rings, with $\kappa$ their common residue field. Let $\mathfrak a=\ker(R\ra R_0)$. Suppose $\mathfrak a\mathfrak m=0$.  Let $\Pi$ be a profinite group, and
\[
\gamma:\Pi \longrightarrow\mathbf{GL}_d(R_0)
\]
a continuous $R_0$-representation of $\Pi$, with $\bar{\gamma}:\Pi\ra\mathbf{GL}_d(\kappa)$ \textcolor{black}{the reduction} modulo $\mathfrak m_0$. We would like to consider deformations of $\gamma$ over $R$, i.e., continuous $R$-representations $\Pi\ra \mathbf{GL}_d(R)$ of $\Pi$ lifting $\rho$.

(1) Let $\tilde{\gamma}:\Pi\rightarrow \mathbf{GL}_d(R)$
be a set-theoretic continuous lifting of $\gamma$. Set
\[
c:\Pi\times \Pi\longrightarrow 1+\mathbf{M}_d(\mathfrak a)\subset \mathbf{GL}_d(R), \quad (g_1,g_2)\mapsto \tilde{\gamma}(g_1g_2)\tilde{\gamma}(g_2)^{-1}\tilde{\gamma}(g_1)^{-1}.
\]
Then the map $c-1$ is a $2$-cocycle of $\Pi$ with coefficients in the $\Pi$-module
\[
\mathbf M_d(\mathfrak a)\simeq \mathbf M_d(\kappa)\otimes_{\kappa}\mathfrak a\simeq \mathrm{ad}(\bar{\gamma})\otimes_{\kappa}\mathfrak a.
\]
The corresponding cohomology class in $H^2(\Pi, \mathrm{ad}(\bar{\gamma}))\otimes \mathfrak a$, written $\mathrm{ob}(\gamma)$ in the sequel, does not depend on the choice of $\tilde{\gamma}$, and $\mathrm{ob}(\gamma)=0$ if and only if $\gamma$ can be deformed to to an $R$-representation of $\Pi$. In other words, $\mathrm{ob}(\gamma)$ is the obstruction for the existence of a deformation of $\gamma$ over $R$.

(2) The set of isomorphism classes of deformations of $\rho$ over $R$ is a torsor under $H^1(\Pi, \mathrm{ad}(\bar{\rho}))\otimes \mathfrak a$. More precisely, let $
\tilde{\gamma} , \tilde{\gamma}':\Pi\rightarrow \mathbf{GL}_d(R)$ be two deformations of $\gamma$ over $R$. Set
\[
e:\Pi\longrightarrow 1+\mathbf{M}_d(\mathfrak a)\subset \mathbf{GL}_d(R), \quad g\mapsto \tilde{\gamma}(g)\tilde{\gamma}'(g)^{-1}.
\]
Then the map $e-1$ is a $1$-cocycle of $\Pi$ with coefficients in $\mathrm{ad}(\bar{\rho})\otimes \mathfrak a$, and the corresponding class in $H^1(\Pi, \mathrm{ad}(\bar{\gamma}))\otimes \mathfrak a$ depends only on the isomorphism classes of $\tilde{\gamma}$ and $\tilde{\gamma}'$. Conversely, let $w\in H^1(\Pi,\mathrm{ad}(\bar{\gamma}))\otimes \mathfrak a$ and write $w=[\tilde e-1]$ with
\[
\tilde e: \Pi\longrightarrow 1+\mathbf{M}_d(\mathfrak a)\subset \mathbf{GL}_d(R).
\]
Then $g\mapsto \tilde{e}\tilde{\gamma}(g)\in \mathbf{GL}_d(R)$ gives another deformation of $\rho$ over $R$, whose isomorphism class depends only on $w$.

(3) Finally, for $\tilde{\gamma}$ a deformation of $\gamma$ over $R$, its group of automorphisms, i.e., $\Pi$-equivariant $R$-linear isomorphisms $R^d\stackrel{\sim}{\ra}R^d$ reducing to $\mathrm{id}_{R_0^d}$ modulo $\mathfrak a$, is naturally identified with $H^0(\Pi, \mathrm{ad}(\bar{\gamma}))\otimes \mathfrak a$.
\end{remark}

Let $\mathfrak a=\ker(\Z/p^{n+1}\Z\ra \Z/p^n\Z)$. Let $\L$ be the (\'etale) $\Z/p^n\Z$-local system corresponding to $\rho$. Write $\bar{\rho}$ and $\overline{\L}$ the reductions modulo $p$ of $\rho$ and $\L$ respectively. As representations of $\pi_1(X,\bar x)$, we have $\mathrm{ad}(\bar{\rho})=\mathrm{End}(\overline{\L}_{\bar x})\simeq \mathcal End(\overline{\L})_{\bar x}$. Furthermore, there is a canonical injective map
\[
H^2(\pi_1(X,\bar x),\mathcal End(\overline{\L})_{\bar x})\simeq H^2(X_{\rm fet},\mathcal End(\overline{\L}))\longrightarrow H^2(X_{\rm et},\mathcal End(\overline{\L})).
\]
In the following, we still denote by $\mathrm{ob}(\rho)$ the image of $\mathrm{ob}(\rho)\in H^2(\pi_1(X,\bar x), \mathrm{ad}(\bar{\rho}))$ via the induced injective map below
\[
H^2(\pi_1(X,\bar x),\mathrm{ad}(\bar{\rho}))\otimes \mathfrak a\longrightarrow H^2(X_{\rm et},\mathcal End(\overline{\L}))\otimes \mathfrak a.
\]
Thus $\mathrm{ob}(\rho)=0$ if and only $\rho$ can be deformed over $\Z/p^{n+1}\Z$, or equivalently, there is a $\Z/p^{n+1}\Z$-local system $\widetilde{\L}$ so that $\widetilde{\L}\otimes_{\Z/p^{n+1}\Z}\Z/p^n\Z\simeq \L$.

Let $(\bar E,\Fil_{tr},\bar{\phi})$ be the reduction modulo $p$ of $(E,\Fil_{tr},\phi)$, which is also the level-zero HDRF on $X$ corresponding to $\bar{\rho}$. So $(\bar E,\bar{\phi})$ is Frobenius-periodic vector bundle of period $1$ on $X$ such that $\overline{\L}\simeq \bar{E}^{\bar{\phi}=1}$. Moreover, the sheaf $\mathcal{E}nd(\bar E)$ of endomorphisms of $\bar E$ is a vector bundle on $X$, endowed with a Frobenius
\[
\Phi: \mathcal{E}nd(\bar E) \longrightarrow \mathcal{E}nd(\bar E), \quad f\mapsto \bar \phi\circ F^*(f)\circ \bar{\phi}^{-1}.
\]
One checks that the canonical map
\[
\mathcal End(\bar E)^{\Phi=1}\longrightarrow \mathcal End(\bar E^{\bar{\phi}=1}), \quad f\mapsto f|_{\bar{E}^{\bar{\phi}=1}}
\]
is an isomorphism, giving thus a short exact sequence in the \'etale topology:
\[
0\longrightarrow \mathcal End(\overline{\L})\longrightarrow \mathcal End(\bar E)\stackrel{1-\Phi}{\longrightarrow} \mathcal {E}nd(\bar E)\longrightarrow 0.
\]
Taking cohomology we get the exact sequence below
\begin{eqnarray*}
0\longrightarrow H^0(X_{\rm et},\mathcal End(\overline{\L}))\otimes \mathfrak a\stackrel{\alpha^0}{\longrightarrow} H^0(X,\mathcal End(\bar E))\otimes \mathfrak a\stackrel{1-\Phi}{\longrightarrow}H^0(X,\mathcal End(\bar E))\otimes \mathfrak a\stackrel{\beta^0}{\longrightarrow} \\
H^1(X_{\rm et},\mathcal End(\overline{\L}))\otimes \mathfrak a\stackrel{\alpha^1}{\longrightarrow} H^1(X,\mathcal End(\bar E))\otimes \mathfrak a\stackrel{1-\Phi}{\longrightarrow}H^1(X,\mathcal End(\bar E))\otimes \mathfrak a\stackrel{\beta^1}{\longrightarrow} \\  H^2(X_{\rm et},\mathcal End(\overline{\L}))\otimes \mathfrak a\stackrel{\alpha^2}{\longrightarrow} H^2(X,\mathcal End(\bar E))\otimes \mathfrak a\stackrel{1-\Phi}{\longrightarrow}H^2(X,\mathcal End(\bar E))\otimes \mathfrak a.
\end{eqnarray*}
On the other hand, like the deformation theory for representations of profinite groups (see Remark \ref{rem:deformation-of-repns}), there is an obstruction class $\mathrm{ob}(E)\in H^2(X,\mathcal End(\bar E))\otimes \mathfrak a $
so that $\mathrm{ob}(E)=0$ if and only if $E$ can be deformed to a vector bundle on $X_{n+1}$. Moreover the set $\mathrm{Def}(E/X_{n+1})$ of isomorphisms classes of deformations of $E$ over $X_n$ is a torsor under $H^1(X,\mathcal End(\bar E))\otimes \mathfrak a$.

\begin{proposition}\label{prop:obstruction-class} Keep the notation above. Then $\alpha^2(\mathrm{ob}(\rho))=\mathrm{ob}(E)$.
\end{proposition}

\begin{proof} Recall first the construction of $\mathrm{ob}(E)\in H^2(X,\mathcal End(\bar E))\otimes \mathfrak a$. In the following, for $U\subset X$ an open subset, we denote by $U_n$ (resp. by $U_{n+1}$) the open subscheme of $X_{n}$ (resp. of $X_{n+1}$) corresponding to the open subset $|U_i|\subset |X|=|X_n|=|X_{n+1}|$. Let $\mathcal U=(U_i)_{i\in I}$ be an open covering such that one can lift $E_i:=E|_{U_{i,n}}$ to a vector bundle $\tilde{E}_i$ on $U_{i,n+1}$, and that for any $i,j\in I$, the two deformations $\tilde{E}_i|_{U_{ij,n+1}}$ and $\tilde{E}_j|_{U_{ij,n+1}}$ of $E|_{U_{ij,n}}$ are isomorphic: for example, choose $\mathcal U$ such that $E_i$ is a free $\mathcal O_{U_{i,n}}$-module, and take $\tilde{E}_i$ a free $\mathcal O_{U_{i,n+1}}$-module lifting $E_i$. For each $i,j\in I$,choose arbitrarily an isomorphism of $\mathcal O_{U_{ij,n+1}}$-modules
\[
f_{ij}:\tilde{E}_i|_{U_{ij,n+1}}\stackrel{\sim}{\longrightarrow}\tilde{E}_j|_{U_{ij,n+1}}, \quad \textrm{so that }f_{ij}|_{U_{ij,n}}=\mathrm{id}_{E|_{U_{ij,n}}}.
\]
Here $U_{ij}:=U_i\cap U_j$. For $i,j,l\in I$, set $U_{ijl}=U_i\cap U_j\cap U_l$ and
\[
g_{ijl}=f_{il}^{-1}|_{U_{ijl,n+1}}\circ f_{jl}|_{U_{ijl,n+1}}\circ f_{ij}|_{U_{ijl,n+1}}: \tilde{E}_{i}|_{U_{ijl,n+1}}\longrightarrow \tilde{E}_i|_{U_{ijl,n+1}}.
\]
Then $
g_{ijl}-\mathrm{id}\in \mathfrak a\cdot \mathrm{End}(\tilde{E}_i|_{U_{ijl,n+1}})\simeq \mathrm{End}(\bar{E}|_{U_{ijl}})\otimes \mathfrak a$,
and the map $(i,j,l)\mapsto g_{ijl}-\mathrm{id}$ is a $2$-cocycle of $\mathcal U$ with values in $\mathcal{E}nd(\bar{E})\otimes \mathfrak a$. The obstruction class $\mathrm{ob}(E)$ is the image of $[g-\mathrm{id}]\in H^2(\mathcal U,\mathcal End(\bar E))\otimes \mathfrak a$ via the natural map below
\[
H^2(\mathcal U, \mathcal End(\bar E))\otimes \mathfrak a\longrightarrow H^2(X,\mathcal End(\bar E))\otimes \mathfrak a,
\]
Since we can use \textcolor{black}{the \'etale topology to compute coherent cohomology}, in the construction of $\mathrm{ob}(E)$ above, one may use equally an \'etale covering $\mathcal U$ of $X$. Similarly, as $\L$ can be locally lifted to a $\Z/p^{n+1}$-local system for the finite \'etale topology on $X$, one may define exactly in the same way an obstruction class
\[
\mathrm{ob}(\L)\in H^2(X_{\rm fet},\mathcal{E}nd(\overline{\L}))\otimes \mathfrak a\hookrightarrow H^2(X_{\rm et},\mathcal{E}nd(\overline{\L}))\otimes \mathfrak a.
\]
As $E\simeq \L\otimes_{\Z/p^n\Z}\mathcal{O}_{X_n}$, we have $\alpha^2(\mathrm{ob}(\L))=\mathrm{ob}(E)$ by functoriality. So to show our proposition, it is enough to compare $\mathrm{ob}(\L)$ and $\mathrm{ob}(\rho)$ in
\[
H^2(\pi_1(X,\bar x), \mathrm{End}(\overline{\L}_{\bar x}))\otimes \mathfrak a\simeq H^2(X_{\rm fet},\mathcal End(\overline{\L}))\otimes \mathfrak a.
\]

Let $Y/X$ be a Galois \'etale covering of Galois group $\Delta:=\mathrm{Aut}(Y/X)$, such that $\L|_Y$ is constant. Let $\bar y$ be a geometric point of $Y$ over $\bar x$. We have $\pi_1(Y,\bar y)\subset \ker(\rho)$, and $\rho$ factors through the morphism $
\pi_1(X,\bar x)\rightarrow \Delta^{\rm op}$, which
sends $g\in \pi_1(X,\bar x)$ to the unique $\sigma^{g}\in \Delta$ such that $g\cdot \bar y=\sigma^{g}(\bar y)$. Moreover, we have
\begin{equation}\label{eq:identification1}
\coprod_{\sigma\in \Delta}Y_{\sigma}=Y\times \Delta\stackrel{\sim}{\longrightarrow} Y\times_X Y, \quad (y,\sigma)\mapsto (y,\sigma(g)),
\end{equation}
with $Y_{\sigma}=Y$ for every $\sigma$. Similarly,
\begin{equation}\label{eq:identification2}
\coprod_{(\sigma,\tau)\in \Delta\times \Delta}Y_{\sigma,\tau}=Y\times \Delta\times \Delta\stackrel{\sim}{\longrightarrow} Y\times_X Y\times_X Y, \quad (y,\sigma,\tau)\mapsto (y,\sigma(y), \tau(\sigma(y))),
\end{equation}
with $Y_{\sigma,\tau}=Y$ for every $(\sigma,\tau)\in \Delta\times\Delta$.
Since $\L|_Y$ is constant, it can be lifted to a constant $\Z/p^{n+1}\Z$-local system $\widetilde{\L}$. Let
\begin{equation} \label{eq:f-tilde}
\tilde f:p_1^*\widetilde{\L}\stackrel{\sim}{\longrightarrow}p^{*}_2\widetilde{\L}
\end{equation}
be an isomorphism of $\Z/p^{n+1}\Z$-modules on $Y\times_X Y$, lifting the canonical isomorphism $p_1^{*}\L\stackrel{\sim}{\ra}p_2^*\L$. Here $p_i:Y\times_XY\ra Y$ for $i=1,2$ are the two projections. Using \eqref{eq:identification1}, $\tilde f$ is identified with a collection of isomorphisms $(\tilde f_{\sigma})_{\sigma\in\Delta}$ on $Y$, with $\tilde f_{\sigma}:\widetilde{\L}\stackrel{\sim}{\rightarrow} \sigma^*\widetilde{\L}$ an isomorphism of $\Z/p^{n+1}\Z$-modules lifting the canonical one $f_{\sigma}:\L_Y\stackrel{\sim}{\ra}\sigma^*\L_Y$ (recall $\L_Y=\L|_Y$). On the other hand, let $g\in \pi_1(X,\bar x)$ with $\sigma(\bar y)=g\bar y$. Since  $Y\ra X$ is \'etale, every finite \'etale $Y$-scheme is finite \'etale over $X$. So $g\in \pi_1(X,\bar x)$ induces an isomorphism \[
\gamma_g:F_{Y,\bar y}\stackrel{\sim}{\longrightarrow}F_{Y,\sigma(\bar y)}
\]
of fiber functors $F_{Y,\bar y}$ and $F_{Y,\sigma(\bar y)}$ of the category $Y_{\rm fet}$ of finite \'etale $Y$-schemes, i.e., a parallel transport from $\bar y$ to $\sigma(\bar y)$. On the other hand, in terms of representation, $\sigma^*\L_Y$ corresponds to the representation (which is in fact trivial)
\[
\pi_1(Y,\bar y)\stackrel{\sigma_*}{\longrightarrow}\pi_1(Y,\sigma(\bar y))\stackrel{\rho_{\bar y}}{\longrightarrow}\mathrm{Aut}(\L_{Y,u(\bar y)}).
\]
Using the parallel transport $\gamma_g$, we obtain an isomorphism of profinite groups
\[
\iota:\pi_1(Y,\sigma(\bar y))=\mathrm{Aut}(F_{Y,\sigma(\bar y)})\stackrel{\sim}{\longrightarrow}\pi_1(Y,\bar y)=\mathrm{Aut}(F_{Y,\bar y}), \quad h\mapsto \gamma_g^{-1}\circ h\circ \gamma_g,
\]
and the commutative square below
\[
\xymatrix{\pi_1(Y,\sigma(\bar y))\ar[r]^{\rho_{\sigma(\bar y)}} \ar[d]_{\iota}& \mathrm{Aut}(\L_{Y,\sigma(\bar y)})\ar[d]^{\phi\mapsto \gamma_g^{-1}\phi\gamma_g} \\ \pi_1(Y,\bar y)\ar[r]^{\rho_{\bar y}} & \mathrm{Aut}(\L_{Y,\bar y})}.
\]
As $\L_Y=\L|_Y$, $\L_{Y,\bar y}$ and $\L_{Y,\sigma(\bar y)}$ are identified with $\L_{\bar x}$, so that the parallel transport $
\gamma_g(\L_Y):\L_{Y,\bar y}\rightarrow \L_{Y,\sigma(\bar y)}$
becomes the $\Z/p^n\Z$-linear map $\rho(g): \L_{\bar x}\ra \L_{\bar x}$. Moreover, viewed as subgroups of $\pi_1(X,\bar x)$, the map $\iota\circ \sigma_*:\pi_1(Y,\bar y)\rightarrow \pi_1(Y,\bar y)$ sends $h$ to $ghg^{-1}$.
Consequently, using the parallel transport $\gamma_g$, the isomorphism $f_{\sigma}:\L_Y\ra \sigma^*\L$ is identified with the $\Z/p^n\Z$-linear map $\rho(g):\L_{\bar x}\ra \L_{\bar x}$ over the morphism $\iota\circ \sigma_*$ above. In a similar way, the isomorphism $\tilde f_{\sigma}:\widetilde{\L}\ra \sigma^*\widetilde{\L}$ can be identified with a $\Z/p^{n+1}\Z$-linear isomorphism $
\rho(g)^{\sim}: \widetilde{\L}_{\bar y}\rightarrow \widetilde{\L}_{\bar y}$
lifting $\rho(g)$. As $\widetilde{\L}$ is constant, the isomorphism $\gamma_g(\widetilde{\L}):\widetilde{\L}_{\bar y}\stackrel{\sim}{\ra}\widetilde{\L}_{\sigma(\bar y)}$ does not depend on the choice of the parallel transport $\gamma_g$ from $\bar y$ to $\sigma(\bar y)$. As a result, $\rho(g)^{\sim}$ depends only on $\sigma$ (recall $g\in \pi_1(X,\bar x)$ is such that $\sigma(\bar y)=g\cdot \bar y$).

Consider the isomorphism $\tilde{f}$ in \eqref{eq:f-tilde}, and the $2$-cocycle
\[
p_{13}^{*}(\tilde f)^{-1}\circ p_{23}^*(\tilde f)\circ p_{12}^*(\tilde f)- \mathrm{id}\in \mathfrak a\mathrm{End}(\widetilde{\L}|_{Y\times Y\times Y})\simeq \mathcal{E}nd(\overline{\L})(Y\times Y\times Y)\otimes \mathfrak a
\]
relative to the \'etale covering $\mathcal U:=\{Y\ra X\}$, with $
p_{12},p_{23},p_{13}:Y\times_X Y\times_X Y\rightarrow Y\times_XY$
the three natural projections. The obstruction class $\mathrm{ob}(\L)$ is by definition the image of this cohomology class through the canonical map
\[
H^2(\mathcal U,\mathcal End(\overline{\L}))\otimes \mathfrak a\longrightarrow H^2(X_{\rm fet},\mathcal End(\overline{\L}))\otimes \mathfrak a.
\]
Using \eqref{eq:identification2}, the restriction to $Y_{\sigma,\tau}$
of $p_{13}^{*}(\tilde f)^{-1}\circ p_{23}^*(\tilde f)\circ p_{12}^*(\tilde f)-\mathrm{id}$ is
\[
\tilde f_{\tau\sigma}^{-1}\circ \sigma^*(\tilde f_{\tau})\circ \tilde f_{\sigma}-\mathrm{id}\in \mathfrak a\mathrm{End}(\widetilde{\L}|_{Y\times Y\times Y}).
\]
Let $g,h\in \pi_1(X,\bar x)$ with $\sigma(\bar y)=g\bar y$ and $\tau(\bar y)=h\bar y$. In terms of representations, the morphism above becomes
\[
\tilde{\rho}(gh)^{-1}\circ \tilde{\rho}(h)\circ \tilde\rho(g)- \mathrm{id}\in \mathfrak a\mathrm{End}(\widetilde{\L}_{\bar y})\simeq \mathrm{End}(\overline{\L}_{\bar x})\otimes \mathfrak a.
\]
Varying $g,h$, we obtain the $2$-cocycle defining $\mathrm{ob}(\rho)$ (Remark \ref{rem:deformation-of-repns}). As a result,
\[
\mathrm{ob}(\L)=\mathrm{ob}(\rho)\in H^2(\pi_1(X,\bar x),\mathrm{End}(\overline{\L}_{\bar x}))\otimes \mathfrak a\simeq H^2(X_{\rm fet},\mathcal End(\overline{\L}))\otimes \mathfrak a.
\]
This completes the proof of our proposition.
\end{proof}

\begin{corollary}\label{obstruction space}Keep the notation above. Suppose moreover that the base field $k$ is algebraically closed and that $X/k$ is proper. Then $\mathrm{ob}(\rho)=0$ if and only if $\mathrm{ob}(E)=0$.
\end{corollary}

\begin{proof} Clearly if $\mathrm{ob}(\rho)=0$ then $\mathrm{ob}(E)=0$. Conversely, suppose $\mathrm{ob}(E)=0$. By Proposition \ref{prop:obstruction-class}, to conclude that $\mathrm{ob}(\rho)=0$, it is enough to check that the map
\[
1-\Phi: H^1(X,\mathcal End(\bar E))\longrightarrow H^1(X,\mathcal End(\bar E))
\]
is surjective. But this is well-known because $\Phi$ is Frobenius-semilinear and the base field $k$ is algebraically closed (\cite[Expos\'e XXII, Proposition 1.2]{SGA7}).
\end{proof}

\begin{remark} Keep the notation above. We have $\mathrm{ob}(\rho)=0\in H^2(X_{\rm et},\mathcal End(\overline{\L}))$
for all representations $\rho:\pi_1(X,\bar x)\ra \mathrm{GL}_n(\Z/p^n\Z)$ in the following two cases:
\begin{enumerate}
\item $X$ is affine connected and smooth over a field of characteristic $p$.
\item $X$ is a proper connected smooth curve over an algebraically closed field of characteristic $p$. 
\end{enumerate}
\textcolor{black}{In fact, one has a slightly more general result as follows. Let $(R,\mathfrak m)\ra (R_0,\mathfrak m_0)$ be a surjective morphism of local Artinian rings as in Remark \ref{rem:deformation-of-repns}, such that the ideal $\mathfrak a=\ker(R\ra R_0)$ of $R$ is annihilated by $\mathfrak m$. Assume that their common residue field $\kappa$ is a \emph{finite field of characteristic $p$}.  Then, for $X$ a scheme satisfying one of the two conditions above, and for any continuous representation $\gamma:\pi_1(X,\bar x)\ra \mathbf{GL}_d(R_0)$, the obstruction 
\[
\mathrm{ob}(\gamma)\in H^2(\pi_1(X,\bar x), \mathrm{ad}(\bar{\gamma}))\otimes_{\kappa} \mathfrak a
\]  
for the existence of a deformation of $\gamma$ over $R$ \emph{vanishes}. Here $\bar{\gamma}$ denotes as usual the reduction modulo $\mathfrak m_0$ of $\gamma$.} Indeed, in both cases, it suffices to show that, for $M$ a finite-dimensional continuous $\mathbb F_p$-representation of $\pi_1(X,\bar x)$, 
\[
H^2(\pi_1(X,\bar x),M)=0.
\]
Write $M$ as the stalk at the geometric base point $\bar x$ of an $\mathbb F_p$-local system $\mathbb M$. Since the natural map below is always injective  
\[
H^2(\pi_1(X, \bar x),M)\longrightarrow H^2(X_{\rm et}, \M),
\]
we reduce our assertion to the well-known result that $H^2(X_{\rm et},\M)=0$ for any $\F_p$-local system $\M$. To see this, write $\M=\mathcal M^{\phi=1}$, with $(\mathcal M,\phi)$ a Frobenius-periodic vector bundle of period $1$ on $X$, giving a short exact sequence of Artin-Schreier-type in the \'etale topology
\[
0\longrightarrow \M\longrightarrow \mathcal M\stackrel{1-\phi}{\longrightarrow} \mathcal M\longrightarrow 0.
\]
Therefore, the assertion in the first case follows since $H^i(X,\mathcal M)=0$ for any $i\geq 1$. In the second case, we have still $H^i(X,\mathcal M)=0$ whenever $i\geq 2$ for the reason of dimension. Furthermore, since $k$ is algebraically closed, the map
\[
1-\phi: H^1(X,\mathcal M)\longrightarrow H^1(X,\mathcal M)
\]
is surjective (\cite[Expos\'e XXII, Proposition 1.2]{SGA7}). As a result, we get equally $H^2(X_{\rm et},\M)=0$ in the second case. \textcolor{black}{Consequently, for $X$ a proper connected smooth curve over an algebraically closed field of characteristic $p$ as in (2), and for an absolute irreducible representation $\gamma_0:\pi_1(X,\bar x)\ra \mathbf{GL}_n(\kappa)$, its deformation functor is pro-representable by a ring of formal power series. See also \cite[\S~3]{dJ} (especially \cite[Lemma 3.11]{dJ}) for the $\ell$-adic analog of this statement.\footnote{\textcolor{black}{We would like to thank one of the referees for drawing our attention to \cite{dJ}.}}}
\end{remark}

Next suppose $\alpha(\mathrm{ob}(\L))=\mathrm{ob}(E)=0$, so that $E$ can be deformed over $X_{n+1}$. Let $E_{n+1}$ be such a deformation: it is a vector bundle on $X_{n+1}$ equipped with an identification $\iota: E_{n+1}|_{X_n}\stackrel{\sim}{\rightarrow}E$. As $E$ is the initial term of a HDRF, we can apply to it the inverse Cartier transform $C_{n+1}^{-1}$, and get a vector bundle $C_{n+1}^{-1}(E_{n+1})$ on $X_{n+1}$, endowed with the composed isomorphism below
\[
C_{n+1}^{-1}(E_{n+1})|_{X_n}\simeq C_n^{-1}(E_{n+1}|_{X_n})\stackrel{F^*(\iota)}{\longrightarrow}C_n^{-1}(E)\stackrel{\phi}{\longrightarrow}E.
\]
In particular, the pair $(C_{n+1}^{-1}(E_{n+1}),\phi\circ F^*(\iota))$ gives a second deformation of $E$ over $X_{n+1}$. Thus there exists a unique $v\in H^1(X,\mathcal End(E))\otimes \mathfrak a$ with
\[
[(C_{n+1}^{-1}(E_{n+1}),\phi\circ F^*(\iota))]=[(E_{n+1},\iota)]+v \in \mathrm{Def}(E/X_{n+1}).
\]
Let $(E_{n+1}',\iota')$ be another deformation of $E$ over $X_{n+1}$, and $w\in H^1(X,\mathcal End(\bar E))\otimes \mathfrak a$ such that
\[
[E_{n+1}',\iota']=[E_{n+1},\iota]+w\in \mathrm{Def}(E/X_{n+1})
\]
Let $v'\in H^1(X,\mathcal End(\bar E))\otimes \mathfrak a$ be the element defined by the equality
\[
[(C_{n+1}^{-1}(E_{n+1}'),\phi\circ F^*(\iota))]=[(E_{n+1}',\iota')]+v' \in \mathrm{Def}(E/X_{n+1}).
\]
\begin{proposition} \label{prop:obstruction-class-when-liftable} Keep the notation and assumption above. In particular, $\mathrm{ob}(E)=0$, i.e., there is no obstruction to deform $E$. Then we have
\begin{enumerate}
\item $\beta^1(v)=\mathrm{ob}(\rho)$; and
\item $v'=v+\Phi(w)-w\in H^1(X,\mathcal End(\bar E))\otimes \mathfrak a$.
\end{enumerate}
\end{proposition}

\begin{proof} (1) Recall briefly the construction of $v\in H^1(X,\mathcal End(\bar E))\otimes \mathfrak a$. As any two deformations of $E$ over $X_{n+1}$ are Zariski-locally isomorphic, one can find an open covering $\mathcal U=(U_i)$ of $X$, and an isomorphism $\phi_i: C_{n+1}^{-1}(E_{n+1})|_{U_i}=C_{n+1}^{-1}(E_{n+1}|_{U_i})\stackrel{\sim}{\rightarrow}E_{n+1}|_{U_i}$ for every $i$,
inducing $\phi|_{U_i}$ modulo $p^n$. Then $v$ is the class of the $1$-cocycle below
\[
(i,j)\mapsto \varphi_{ij}:=\phi_j|_{U_{ij}}\circ \phi_i^{-1}|_{U_{ij}}-\mathrm{id}\in \mathfrak a\mathrm{End}(E_{n+1}|_{U_{ij}})\simeq \mathrm{End}(\bar E|_{U_{ij}})\otimes \mathfrak a.
\]
In the following we consider a special case: for the general case one needs a further \'etale covering $U_{ij}^w$ of $U_{ij}$ for each $i,j$ (so we need the notion of hypercoverings of $X$). Suppose that there exist $f_{ij}\in \mathrm{End}(\bar E|_U{_{ij}})\otimes \mathfrak a$ for all $i,j$, with $\varphi_{ij}=f_{ij}-\Phi(f_{ij})=f_{ij}-\bar{\phi}\circ F^*(f_{ij})\circ \bar{\phi}^{-1}$,
or equivalently,
\[
\phi_j|_{U_{ij}}\circ \phi_i^{-1}|_{U_{ij}}  =  (\mathrm{id}+\bar{\phi}\circ F^*(f_{ij})\circ \bar{\phi}^{-1})^{-1}\circ (\mathrm{id}+f_{ij})\in \mathrm{id}+\mathfrak a\mathrm{End}(E_{n+1}|_{U_{ij}}).
\]
Therefore $
C_{n+1}^{-1}(\mathrm{id}+f_{ij})\circ (\phi_j|_{U_{ij}})=(\mathrm{id}+f_{ij})\circ (\phi_i|_{U_{ij}})$. As a result, the morphism $
\mathrm{id}+f_{ij}\in \mathrm{id}+\mathrm{End}(\bar E|_{U_{ij}})\otimes \mathfrak a\simeq \mathrm{id}+\mathfrak a\mathrm{End}(E_{n+1}|_{U_{ij}})$
gives an isomorphism of $1$-periodic HDRFs
\[
\mathrm{id}+ f_{ij}: (E_{n+1}|_{U_{ij}},\Fil_{tr},\phi_i)\stackrel{\sim}{\longrightarrow} (E_{n+1}|_{U_{ij}},\Fil_{tr},\phi_j).
\]
Let $\widetilde{\L}_i=(E_{n+1}|_{U_i})^{\phi_i=1}$ which is a $(\Z/p^{n+1}\Z)$-local system lifting $\L=E^{\phi=1}$ on $U_i$. So we deduce an isomorphism $\tilde f_{ij}: \widetilde{\L}_i|_{U_{ij}}\stackrel{\sim}{\rightarrow} \widetilde{\L}_j|_{U_{ij}}$ of $\Z/p^{n+1}\Z$-local systems.
By definition $\beta^1(v)$ is the image in $H^2(X_{\rm et},\mathcal End(\overline{\L}))\otimes \mathfrak a$ of the $2$-cocycle
\[
(i,j,l)\mapsto -f_{il}+f_{ij}+f_{jl}=\tilde{f}_{il}^{-1}\circ \tilde{f}_{jl}\circ \tilde{f}_{ij}-\mathrm{id} \in \mathfrak a\mathrm{End}(\widetilde{\L}_i|_{U_{ijl}})\simeq \mathrm{End}(\overline{\L}|_{U_{ijl}})\otimes \mathfrak a,
\]
which is nothing but the class $\mathrm{ob}(\L)=\mathrm{ob}(\rho)$.

(2) Let $\mathcal U=(U_i)_{i\in I}$ be an open covering of $X$ such that for each $i$ there exists an isomorphism $
g_i: E_{n+1}'|_{U_i}\stackrel{\sim}{\rightarrow}E_{n+1}|_{U_i}$
of vector bundles reducing identity on $E|_{U_i}$. Then $w$ is the image of the following $1$-cocycle
\[
(i,j)\mapsto g_j|_{U_{ij}}\circ g_i^{-1}|_{U_{ij}}-\mathrm{id}\in \mathfrak a\mathrm{End}(E_{n+1}|_{U_{ij}})\simeq \mathrm{End}(E_{n+1}|_{U_{ij}})\otimes \mathfrak a.
\]
Shrinking $\mathcal U$ if necessary, we suppose as in (1) that there exists for each $i$ an isomorphism
\[
\phi_i: C_{n+1}^{-1}(E_{n+1})|_{U_i}=C_{n+1}^{-1}(E_{n+1}|_{U_i})\stackrel{\sim}{\longrightarrow}E_{n+1}|_{U_i}
\]
inducing $\phi|_{U_i}$ modulo $\mathfrak a$. Then one deduces an isomorphism
\[
\phi_i'=g_i^{-1}\circ \phi_i\circ C_{n+1}^{-1}(g_i):C_{n+1}^{-1}(E_{n+1}'|_{U_i})\stackrel{\sim}{\longrightarrow}E_{n+1}'|_{U_i}.
\]
which induces equally $\phi|_{U_i}$ modulo $\mathfrak a$. So
\[
\phi_j'|_{U_{ij}}\circ {\phi_i'}^{-1}|_{U_{ij}}=g_j^{-1}\circ \phi_j\circ C_{n+1}^{-1}(g_jg_i^{-1})\circ \phi_i^{-1}\circ g_i
\]
yielding
\[
g_j\circ (\phi_j'\circ {\phi_i'}^{-1})\circ g_j^{-1}=\phi_j\circ C_{n+1}^{-1}(g_jg_i^{-1})\circ \phi_j^{-1}\circ (\phi_j\circ \phi_i^{-1})\circ (g_i\circ g_j^{-1}).
\]
Varying $i,j$ we find $v'=v+\Phi(w)-w\in H^1(X,\mathcal End(\bar E))\otimes \mathfrak a$, as wanted.
\end{proof}

Finally, suppose $\mathrm{ob}(\rho)=0$ and let $
\tilde{\rho}, \tilde{\rho}':\pi_1(X,\bar x)\rightarrow \mathbf{GL}_d(\Z/p^{n+1}\Z)$
be two deformations of $\rho$ over $\Z/p^{n+1}\Z$. By Remark \ref{rem:deformation-of-repns}, there exists a unique
\[
w_{\rm Gal}\in H^1(\pi_1(X,\bar x),\mathrm{End}(\bar{\rho}))\otimes \mathfrak a\hookrightarrow H^1(X_{\rm et},\mathcal End(\overline{\L}))\otimes \mathfrak a
\]
with $[\tilde{\rho}']=[\tilde{\rho}]+w_{\rm Gal}$. 
Let $(E_{n+1},\Fil_{tr},\phi_{n+1})$ and $(E_{n+1}',\Fil_{tr},\phi_{n+1}')$ be the $1$-periodic HDRF corresponding to $\tilde{\rho}$ and $\tilde{\rho}'$ respectively. Thus $E_{n+1}$ and $E_{n+1}'$ are two deformations of $E$ over $X_{n+1}$. Let
\[
w_{\rm coh}\in H^1(X,\mathcal End(\bar{E}))\otimes \mathfrak a
\]
be the unique element such that $[E_{n+1}']=[E_{n+1}]+w_{\rm coh}\in \mathrm{Def}(E/X_{n+1})$.

\begin{proposition}\label{prop:def-spaces-and-auto} Keep the notations above. In particular, $\tilde{\rho}$ and $\tilde{\rho}'$ are two deformations of $\rho$ over $\Z/p^{n+1}\Z$, with corresponding level-zero HDRFs $(E_{n+1},\Fil_{tr},\phi_{n+1})$ and $(E_{n+1}',\Fil_{tr},\phi_{n+1}')$ respectively. Then
\begin{enumerate}
\item we have $\alpha^1(w_{\rm Gal})=w_{\rm coh}$;
\item if moreover $w_{\rm coh}=0$, let $f:E_{n+1}'\stackrel{\sim}{\ra}E_{n+1}$ be an isomorphism of vector bundles reducing to $\mathrm{id}_{E}$ modulo $\mathfrak a$, and set
\[
g= f\circ \phi_{n+1}'\circ C_{n+1}^{-1}(f^{-1})\circ \phi_{n+1}^{-1}-\mathrm{id} \in \mathfrak a\mathrm{End}(E_{n+1})\simeq \mathrm{End}(\bar E)\otimes \mathfrak a,
\]
then $w_{\mathrm{Gal}}=\beta^0(g)$; and
\item the map $\alpha^0$ identifies the automorphism group of the deformations $\tilde{\rho}$ with $H^0(X,\mathcal End(\bar E))^{\Phi=1}\otimes \mathfrak a$, that is, the group of automorphisms of $E_{n+1}\in \mathrm{Def}(E/X_{n+1})$ that commute with $\phi_{n+1}$.
\end{enumerate}
\end{proposition}

\begin{proof} The proof is similar to the ones above, so we omit the details here.
\end{proof}

\begin{corollary}\label{deformation space}Keep the notations of Proposition \ref{prop:def-spaces-and-auto}. Suppose moreover that the base field $k$ is algebraically closed and that $X/k$ is proper and smooth. Then the two representations $\tilde{\rho}$ and $\tilde{\rho}'$ are isomorphic as deformations of $\rho$, or equivalently, the two HDRFs $(E_{n+1},\Fil_{tr},\phi_{n+1})$ and $(E_{n+1},\Fil_{tr},\phi_{n+1}')$ are isomorphic as deformations of $(E,\Fil_{tr},\phi)$, if and only if $E_{n+1}$ and $E_{n+1}'$ have the same isomorphism class in $\mathrm{Def}(E/X_{n+1})$. In particular the natural map
\[
\left\{\begin{array}{c}\textrm{isomorphism classes}\\ \textrm{of deformations of } \\ (E,\Fil_{tr},\phi) \textrm{ over }X_{n+1}\end{array}\right\} \longrightarrow \mathrm{Def}(E/X_{n+1}), \quad (E_{n+1},\Fil_{tr},\phi_{n+1})\mapsto E_{n+1}
\]
is injective and identifies the first set, or equivalently the set of isomorphism classes of deformations of $\rho$ over $\Z/p^{n+1}\Z$, with a subset of $\mathrm{Def}(E/X_{n+1})$, and the latter is a torsor under $H^1(X,\mathcal{E}nd(E))^{\Phi=1}\otimes \mathfrak a$. Moreover this identification is compatible with the natural isomorphism below
\[
H^1(X,\mathcal{E}nd(\overline{\L}))\otimes \mathfrak a \stackrel{\sim}{\longrightarrow} H^1(X,\mathcal{E}nd(\bar{E}))^{\Phi=1}\otimes \mathfrak a.
\]
\end{corollary}

\section{Galois action on representations of $\pi_1$ in the finite field case}\label{section:Galois-action}

\if false
Let $k$ be a field, and $k^{sep}$ a separable closure of $k$. Write $G_k=\mathrm{Gal}(k^{sep}/k)$ for the absolute Galois group of $k$. Let $X$ be a geometrically connected smooth variety over $k$. Set $\bar X=X\otimes k^{sep}$. Then each $\sigma\in G_k$ induces an automorphism
\[
1\otimes \sigma: \bar X\stackrel{\sim}{\longrightarrow} \bar X.
\]
of the scheme $\bar X$, and thus an isomorphism by the functoriality of $\pi_1$
\begin{equation}\label{eq:auto-of-geo-pi1}
\pi_1^{geo}(\bar X)\stackrel{\sim}{\longrightarrow}\pi_1^{geo}(\bar X).
\end{equation}
Because of the different choices of geometric base point, the isomorphism \eqref{eq:auto-of-geo-pi1} is only well-defined up to an inner automorphism of $\pi_1^{geo}(X)$. Let $\alpha(\sigma)$ be the image of the automorphism \eqref{eq:auto-of-geo-pi1} in the group
\[
\Out(\pi_1^{geo}(X)):=\Aut(\pi_1^{geo}(X))/\mathrm{Inn}(\pi_1^{geo}(X))
\]
of outer automorphisms of $\pi_1^{geo}(X)$. We have
\[
\alpha(\sigma\tau)=\alpha(\tau)\alpha(\sigma), \quad \forall \ \sigma,\tau\in G_k.
\]
In this way we get a group homomorphism
\[
\alpha=\alpha_{X/k}:G_k^{\rm op}\longrightarrow \Out(\pi_1(\bar X)),
\]
which we refer to as the \emph{outer Galois (right) action} of $G_k$ on $\pi_1^{geo}(X)$. Here for $G$ a group, we denote by $G^{\rm op}$ its \emph{opposite group}, i.e.,  the group whose underlying set is the same as that of $G$ but with the new multiplication $g*h = h\cdot g$, i.e. multiply as if we were in $G$ but reverse the order. It is a fundamental question in arithmetic geometry to understand this outer Galois action $\alpha_{X/k}$.

Let $R$ be a topological commutative ring. Let
\[
\rho:\pi_1^{geo}(X)\longrightarrow \mathrm{GL}_r(R)
\]
be an object in the category $\Rep_R(\pi_1^{geo}(X))$ of $R$-linear representations of $\pi_1^{geo}(X)$ in finite free $R$-modules. Let $\sigma\in G_k$, and $\tilde{\sigma}\in \Aut(\pi_1^{geo}(X))$ a lifting of $\alpha(\sigma)\in \Out(\pi_1^{geo}(X))$. Consider the representation
\[
\rho\circ \tilde{\sigma}: \pi_1^{geo}(X)\longrightarrow \mathrm{GL}_r(R).
\]
Since the difference between two liftings of $\alpha(\sigma)$ is an inner automorphism of $\pi_1^{geo}(X)$, the isomorphism class $[\rho\circ \tilde{\sigma}]$ of $\rho\circ \tilde{\sigma}$ does not depend on the choice of $\tilde{\sigma}$. By abuse of notation, let us denote this isomorphism class by $\rho\circ \alpha(\sigma)$.  Then, the outer Galois representation \eqref{eq:auto-of-geo-pi1} induces a right action of $G_k$ by $\alpha$ on (the set of isomorphism classes of)
$\Rep_R(\pi_1^{geo}(X))$ by sending $\rho$ to
$$
\rho^{\sigma}:=\rho\circ \alpha(\sigma).
$$
On the other hand, if we are given an action of $G_k$ on $R$, i.e., a continuous morphism $\phi: G_k\to \Aut(R)$, there is an induced left action of $G_k$ on $\Rep_R(\pi_1^{geo}(X))$ by mapping $\rho$ to
\begin{equation}
^{\sigma}\rho:=\beta(\sigma)\circ \rho,
\end{equation}
where $\beta(\sigma)$ is the induced Galois action of $\phi(\sigma)$ on $\GL_r(R)$. As strongly suggested by the $p$-adic Hodge theory, one hopes that, when the pair $(R,\phi)$ as above is chosen properly, there exist some nice algebro-geometric parametrization of the category $\Rep_R(\pi_1^{geo}(X))$ (in terms of certain Higgs bundles for example) and some coupling between these two Galois actions, so that we can deduce knowledge of outer Galois representation \eqref{eq:outer} from the diagonal action
\[
\Rep_R(\pi_1^{geo}(X))\ni\rho\mapsto \sigma^{*}(\rho):={}^{\sigma^{-1}}\rho ^{\sigma}
\]
of $G_k$ on representations of the geometric fundamental groups.

\newpage

\fi

Let $p$ be a prime number, and $\bar \F_p$ a fixed algebraic closure of $\F_p$, a finite field with $p$ elements. Let $k\subset \bar{\mathbb F}_p$ be a finite field. Let $n\in \mathbb Z_{>0}\cup\{\infty\}$, and $X_n$ a smooth connected scheme formal over $W_n:=W_n(k)$, with $X=X_n\otimes_{W}k$ its closed fiber. Set $\bar X_n:=X_n\hat{\otimes}_{W_n(k)}W_n(\bar{\mathbb F}_p)$. So $\bar X_1=X\otimes_k \bar{\mathbb F}_p$, and $\pi_1(\bar X_1)=\pi_1^{geo}(X)$. Moreover, every element $\sigma\in \mathrm{Gal}(\bar{\F}_p/k)$ defines two automorphisms
\[
\bar X_n\stackrel{\sim}{\longrightarrow} \bar X_n, \quad \textrm{and}\quad \bar X_1\stackrel{\sim}{\longrightarrow} \bar X_1,
\]
both written by $1\otimes \sigma$ in the sequel. The latter gives also an automorphism
\begin{equation}\label{eq:tilde-sigma}
\tilde{\sigma}=(1\otimes \sigma)_*:\pi_1^{geo}(X)\stackrel{\sim}{\longrightarrow}\pi_1^{geo}(X).
\end{equation}
which is only well-defined up to an inner automorphism of $\pi_1^{geo}(X)$ (since we ignore the choice of base points in the definition of $\pi_1^{geo}(X)$). Let $f\geq 1$ be an integer. 

\textcolor{black} {
\begin{construction}[$\rho^{\sigma}$ and ${}^{\sigma}\rho$]Let $
\rho: \pi_1^{geo}(X)\rightarrow \mathrm{GL}_r(W_n(\mathbb F_{p^f}))$
be a continuous representation of $\pi_1^{geo}(X)$. One can attach to $\rho$ another two continuous representations of $\pi_1^{geo}(X)$ as follows.  \\
\indent (1) Composing $\rho$ with $\tilde{\sigma}$ in \eqref{eq:tilde-sigma}, we get a second representation
\begin{equation*}\label{eq:primary-action}
\rho\circ \tilde{\sigma}: \pi_1^{geo}(X)\stackrel{\tilde{\sigma}}{\longrightarrow} \pi_1^{geo}(X)\stackrel{\rho}{\longrightarrow} \mathrm{GL}_r(W_n(\mathbb F_{p^f}))
\end{equation*}
of $\pi_1^{geo}(X)$, whose isomorphism class only depend on $\sigma$: recall that the difference between two choices of the automorphism $\tilde{\sigma}$ is an inner automorphism of $\pi_1^{geo}(X)$. By abuse of notation, let us denote the representation $\rho\circ \tilde{\sigma}$ above by $\rho^{\sigma}$. \\
\indent (2) On the other hand, $\sigma$ acts on $\F_{p^f}\subset\bar \F_p$ and thus on $\mathrm{GL}_r(W_n(\F_{p^f}))$, so from $\rho$ we can deduce another representation of $\pi_1^{geo}(X)$, denoted by ${}^{\sigma}\rho$ in the sequel:
\[
\pi_1^{geo}(X)\stackrel{\rho}{\longrightarrow} \mathrm{GL}_r(W_n(\mathbb F_{p^f})) \stackrel{\sigma}{\longrightarrow}\mathrm{GL}_r(W_n(\mathbb F_{p^f})).
\]
\end{construction}}

Let 
\[
\rho: \pi_1^{geo}(X)\longrightarrow \mathrm{GL}_r(W_n(\mathbb F_{p^f}))
\] 
be a representation of $\pi_1^{geo}(X)$. By Theorem \ref{equivalence in HT zero case} and Corollary \ref{cor:HT-equivalence-in-formal-case}, it corresponds to an $f$-periodic HDRF of level zero on $X_n$:
\[
(E,0, \Fil_{tr},\ldots, \Fil_{tr},\phi).
\]
We would like to express $\rho^{\sigma}, {}^{\sigma}\rho$ and ${}^{\sigma}\rho^{\sigma}:={}^{\rho}(\rho^{\rho})$ in terms of $(E,0, \Fil_{tr},\ldots, \Fil_{tr},\phi)$. Recall that, if we set $(E_1,0):=\Gr_{\Fil_{tr}}(C_n^{-1}(E,0))$, the tuple
\[
(E_1,0,\Fil_{tr},\ldots, \Fil_{tr},C_n^{-1}(\phi))
\]
is an $f$-periodic HDRF of level zero, called the \textit{shift} of $(E,0,\Fil_{tr},\ldots, \Fil_{tr},\phi)$: see the paragraph before \cite[Corollary 3.11]{LSZ}. Observe that once we shift $f$-times an $f$-periodic HDRF, we get an $f$-periodic HDRF isomorphic to the one that we start with.

\begin{proposition}\label{right action in the finite field case}
Keep the notation as above.
\begin{enumerate}
\item The $f$-periodic HDRF of level zero on $\bar X_n$ corresponding to $^{\sigma}\rho^{\sigma}$ is given by $(1\otimes \sigma)^*(E,0,\Fil_{tr},\ldots,\Fil_{tr},\phi)$.
\item \emph(see also \cite[Corollary 3.11]{LSZ}\emph) Let $\sigma_0\in \Aut(\bar{\mathbb F}_p)$ be the Frobenius automorphism (sending an element of $\bar{\mathbb F}_p$ to its $p$-th power). Then, the periodic bundle corresponding to $^{\sigma_0}\rho $ is the shift of $(E,0,\Fil_{tr},\ldots,\Fil_{tr},\phi)$.
\end{enumerate}
\end{proposition}

\begin{proof} It suffices to consider the case where $n\in \mathbb Z_{>0}$: the case where $n=\infty$ is obtained by taking limits. Let $M:=W_{n}(\mathbb F_{p^f})^r$, endowed with the action of $\pi_1^{geo}(X)$ given by $\rho$. Let $Y\ra \bar X_n$ be a finite \'etale Galois cover such that $\rho$ factors through the quotient $\pi_1^{geo}(X)\twoheadrightarrow \Aut(Y/\bar X_n)^{\rm op}$. So $\Aut(Y/\bar X_n)$ acts on the right on $M$. From the proof of Theorem \ref{equivalence in HT zero case},  the vector bundle $E$ is the quotient of $M\otimes_{W_n(\mathbb F_{p^{f}})}\mathcal O_Y$ by the following right action of $\Aut(Y/\bar X_n)$:
\[
(m\otimes a)\cdot \gamma:= (m\cdot \gamma)\otimes (\gamma^{*}a), \quad \forall m\in M, \ a\in \mathcal O_Y, \ \gamma\in \Aut(Y/\bar X_n).
\]

(1) By the choice of the Galois cover $Y/\bar X_n$, the representation ${}^{\sigma}\rho^{\sigma}$ factors through $\pi_1^{geo}(X)\twoheadrightarrow \Aut(Y'/\bar X_n)^{\rm op}$, with $Y'\ra \bar X_n$ the base change of $Y\ra \bar X_n$ by the morphism $\sigma:\bar X_n\ra \bar X_n$. Let $(E',0,\Fil_{tr},\ldots, \Fil_{tr},\phi')$ be the $f$-periodic HDRF of level zero corresponding to ${}^{\sigma}\rho^{\sigma}$. Then $E'$ is the quotient of
\[
\left(M\otimes_{W_n(\mathbb F_{p^f}),\sigma}W_n(\mathbb F_{p^f})\right)\otimes_{W_n(\mathbb F_{p^f})}\mathcal O_{Y'}
\]
by a similar right action of $\Aut(Y'/\bar X_n)$ as above. So $E'=\sigma^{*}E$. Furthermore, let $U\subset X_n$ be an open subset of $X_n$ endowed with a Frobenius lifting $F=F_U$, then $\bar F:=F\otimes \sigma_0$ is a Frobenius lifting on $\bar U:=U\otimes_{W_n(k)}W_n(\bar F_p)$, where $\sigma_0$ denotes the Frobenius on $W_n(\bar F_p)$. Write $Y_U$ (resp. $Y_U'$) be the open of $Y$ (resp. of $Y'$) above $\bar U\subset \bar X_n$. Since $Y_U\ra \bar U$ (resp. $Y_U'\ra \bar U$) is finite \'etale, the Frobenius lifting $\bar F$ on $\bar U$ lifts uniquely to $Y_U$ (resp. to $Y_U'$). By the uniqueness, the horizontal morphisms in the diagram below are compatible with Frobenius on both sides:
\[
\xymatrix{Y_U'\ar[r]\ar[d] & Y_U\ar[d] \\ \bar U\ar[r]^{\sigma} & \bar U}
\]
Thus the Frobenius periodic vector bundle $(E_U',\phi'_U)$ is the pullback by $1\otimes \sigma$ of $(E_U,\phi_U)$. Hence $(E',0,\Fil_{tr},\ldots, \Fil_{tr},\phi')$ is the pullback of $(E,0,\Fil_{tr},\ldots, \Fil_{tr},\phi)$.

(2): The proof is similar to the one above. Let $(E'',0,\Fil_{tr},\ldots, \Fil_{tr},\phi'')$ be the HDRF of level zero corresponding to ${}^{\sigma_0}\rho$. Then $E''$ is the quotient of
\[
\left(M\otimes_{W_n(\mathbb F_{p^f}),\sigma_0}W_n(\mathbb F_{p^f})\right)\otimes_{W_n(\mathbb F_{p^f})}\mathcal O_{Y}
\]
by the induced right action of $\Aut(Y/\bar X_n)$. Let $U\subset X_n$ be a small open subset endowed with a Frobenius lifting $F$. Thus $\bar U=U\otimes_{W_n(k)}W_n(\bar F_p)$ is equipped with the Frobenius lifting $\bar F=F\otimes \sigma_0$. Like above, the open subset $Y_U$ of $Y$ above $\bar U\subset \bar X_n$ has also a Frobenius lifting, written still by $\bar F$. Then, we have a natural isomorphism of $\mathcal O_Y$-modules
\[
\left(M\otimes_{W_n(\mathbb F_{p^f}),\sigma_0}W_n(\mathbb F_{p^f})\right)\otimes_{W_n(\mathbb F_{p^f})}\mathcal O_{Y}\stackrel{\sim}{\longrightarrow}\left(M\otimes_{W_n(\mathbb F_{p^f})}\mathcal O_{Y_U}\right)\otimes_{\mathcal O_{Y_U},\bar F}\mathcal O_{Y_U}.
\]
In particular, $E''|_{U}\stackrel{\sim}{\ra}\bar F^*(E|_U)$. If $F'$ is a second Frobenius lifting on $U$, with $\bar F'=F\otimes \sigma_0$ the induced Frobenius lifting on $\bar U$, we obtain similarly an isomorphism $E''|_U\stackrel{\sim}{\ra} \bar {F'}^*(E|_U)$. The resulting isomorphism $F^*(E|_U)\stackrel{\sim}{\ra} {F'}^*(E|_{U})$ is the one defined by using the canonical connection on $E$. In other words, there is a natural isomorphism $E''\stackrel{\sim}{\ra} E_1$ of vector bundle on $\bar X_n$ (here we use the notations in the diagram \eqref{eq:HDRF-over-Xn}). Moreover, if we denote by $\phi_F$ the Frobenius on $E|_U$ relative to $\bar F$, then the corresponding Frobenius periodic vector bundle $(E_U'',\phi_{F}'')$ is the pullback by $\bar F$ of $(E_U,\phi_F)$, from where one deduces that the HDRF $(E'',0,\Fil_{tr},\ldots, \Fil_{tr},\phi'')$ is the same as $(E_1,0,\Fil_{tr},\ldots, \Fil_{tr},C_n^{-1}(\phi))$, that is, the shift of $(E,0,\Fil_{tr},\ldots,\Fil_{tr},\phi)$.
\end{proof}


\begin{corollary} \label{main result} Keep the notations as above. Assume that $k\subset \bar{\mathbb F}_p$ is a finite field of $q=p^m$ elements. Let $(E,0,\Fil_{tr},\ldots, \Fil_{tr},\phi)$ be an $f$-periodic HDRF of level zero on $X_n$, with $\rho$ the corresponding $W_{n}(\mathbb F_{p^f})$-representation of $\pi_1^{geo}(X)$. Let $\sigma\in \mathrm{Gal}(\bar{\mathbb F}_p/k)$ be the topological generator sending $a\in \bar{\mathbb F}_p$ to $a^{q}$.
Then, the periodic HDRF corresponding to $\rho^{\sigma}$ is $(1\otimes\sigma)^*(E,0,\Fil_{tr},\ldots,\Fil_{tr},\phi)$ shifted $Nf-m$ times, where $N\in \mathbb N$ is \textcolor{black}{any} integer such that $Nf-m\geq 0$: \textcolor{black}{recall that once we shift $f$-times an $f$-periodic HDRF, we get an $f$-periodic HDRF isomorphic to the one that we start with.}
\end{corollary}

\begin{remark}\label{rem:anabelian-approach}Let $k$ be a field, and $k^{sep}$ a separable closure of $k$. Let $X$ be a geometrically connected smooth variety over $k$. Set $\bar X=X\otimes k^{sep}$. As in the beginning of this section, each $\sigma\in \mathrm{Gal}(k^{sep}/k)$ induces an automorphism
\begin{equation*}
\pi_1^{geo}(\bar X)\stackrel{\sim}{\longrightarrow}\pi_1^{geo}(\bar X),
\end{equation*}
whose image, written $\alpha(\sigma)$, in the group
\[
\Out(\pi_1^{geo}(X)):=\Aut(\pi_1^{geo}(X))/\mathrm{Inn}(\pi_1^{geo}(X))
\]
of outer automorphisms of $\pi_1^{geo}(X)$ is well-defined. We have
\[
\alpha(\sigma\tau)=\alpha(\tau)\alpha(\sigma), \quad \forall \ \sigma,\tau\in G_k.
\]
In this way we get a group homomorphism
\[
\alpha=\alpha_{X/k}:G_k^{\rm op}\longrightarrow \Out(\pi_1(\bar X)),
\]
which we refer to as the \emph{outer Galois (right) action} of $G_k$ on $\pi_1^{geo}(X)$. Here for $G$ a group, we denote by $G^{\rm op}$ its \emph{opposite group}, i.e.,  the group whose underlying set is the same as that of $G$ but with the new multiplication $g*h = h\cdot g$, i.e. multiply as if we were in $G$ but reverse the order. It is a fundamental question in arithmetic geometry to understand this outer Galois action $\alpha_{X/k}$. The reformulation of the Galois action on the representations of geometric fundamental groups in terms of HDRFs might shed light on the study of the outer Galois representations of algebraic varieties defined over finite fields. \end{remark}


\begin{thebibliography}{X-X00}
\bibitem{Crew} R. Crew, \textit{$F$-isocrystals and $p$-adic representations}, Proceedings of Symposia in Pure Mathematics, 46 (1987), 111-138.

\bibitem{SGA7} P. Deligne, N. Katz, \textit{Groupes de monodromie en géométrie algébrique}, Séminaire de Géométrie Algébrique du Bois-Marie 1967-1969, Lecture Notes in Mathematics, 340, Springer-Verlag, 1973.

\bibitem{Fa1}
  G. Faltings, \textit{Crystalline cohomology and $p$-adic Galois-representations}, Algebraic analysis, geometry, and number theory (Baltimore, MD, 1988), 25-80, Johns Hopkins Univ. Press, Baltimore, MD, 1989.

\bibitem{Fa2}
G. Faltings, \textit{Integral crystalline cohomology over very ramified
valuation rings}, Journal of the AMS, 12 (1999), no. 1, 117-144.


\bibitem{Fa3}
G. Faltings, \textit{A $p$-adic Simpson correspondence}, Advances in
Mathematics, 198 (2005), 847-862.

\bibitem{dJ}
A.J. de Jong, \textit{A conjecture on arithmetic fundamental groups}, Israel Journal of Mathematics, 121  (2001), 61-84.

\bibitem{Katz70} N. Katz, \textit{Nilpotent connections and the monodromy theorem: application of a result of Turrittin}, Publ. Math. Inst. Hautes \'{E}tudes Sci., 39 (1970), 175-232.

\bibitem{Katz71}
N. Katz, \textit{Travaux de Dwork}, S\'{e}minaire Bourbaki, 24\`{e}me ann\'{e}e (1971/1972), Exp. No. 409, pp. 167-200, Lecture Notes in Mathematics, 317,
Springer, Berlin, 1973.

\bibitem{Katz73}
N. Katz, \textit{$p$-adic properties of modular schemes and modular forms},
Lecture Notes in Mathematics, 350, Springer, Berlin, 1973.


\bibitem{KYZ}
R. Krishnamoorthy, J. Yang, K. Zuo, \textit{Deformation theory of periodic Higgs-de Rham flows}, arXiv:2005.00579.


\bibitem{LSZ}
G.-T. Lan, M. Sheng, K. Zuo, \textit{Semistable Higgs bundles, periodic Higgs bundles and representations of algebraic fundamental groups}, Journal of European Mathematical Society, 21 (2019), 3053-3112.

\bibitem{LS}
H. Lange, U. Stuhler, \textit{Vektorb\"{u}ndel auf Kurven und Darstellungen der algebraischen Fundamentalgruppe}, Math. Z., 156 (1977), 73-83.

\bibitem{Ma} 
B. Mazur, \textit{Deforming Galois representations}, in Galois groups over $\mathbb Q$, Y. Ihara, K. Ribet, J.-P. Serre eds., MSRI Publ. 16 (1987), Springer, New York-Belin, 1989.

\bibitem{OV}
A. Ogus, V. Vologodsky, \textit{Nonabelian Hodge theory in characteristic $p$}, Publ. Math. Inst. Hautes \'{E}tudes Sci., 106 (2007), 1-138.












\end{thebibliography}
\end{document}